\documentclass[letterpaper, 11pt,reqno]{amsart}
\usepackage{amsmath}
\usepackage{amsthm}
\usepackage{amssymb}
\usepackage{amsfonts}
\usepackage{cases}
\usepackage{graphicx}
\usepackage{xcolor}
\usepackage[margin=1in]{geometry}
\usepackage[utf8]{inputenc}
\usepackage[margin=1in]{geometry}
\usepackage{verbatim}
\usepackage{caption,subcaption}
\usepackage[colorlinks,citecolor=blue,urlcolor=blue]{hyperref}
\usepackage{bm}
\usepackage{enumerate}

\usepackage{algorithm}
\usepackage{algpseudocode}
\usepackage[normalem]{ulem}
\makeatletter
\renewcommand{\fnum@algorithm}{\fname@algorithm}

\numberwithin{equation}{section}

\newtheorem{Remark}{Remark}[section]
\newtheorem{Theorem}{Theorem}[section]
\newtheorem{Lemma}{Lemma}[section]
\newtheorem{Proposition}{Proposition}[section]
\newtheorem{Corollary}{Corollary}[section]

\newtheorem{Condition}{Condition}[section]
\newcommand{\be}{\begin{equation}}
\newcommand{\ee}{\end{equation}}
\newcommand{\bee}{\begin{equation*}}
\newcommand{\eee}{\end{equation*}}
\newcommand{\bi}{\begin{itemize}}
\newcommand{\ei}{\end{itemize}}

\usepackage{algorithm}
\usepackage{algpseudocode}
\usepackage{soul}

\def \C{\mathbb{C}}
\def \D{\mathbb{D}}
\def \E{\mathbb{E}}

\def \M{\mathbb{M}}
\def \N{\mathbb{N}}
\def \P{\mathbb{P}}

\def \R{\mathbb{R}}
\def \S{\mathbb{S}}

\def \Z{\mathbb{Z}}

\def \X{\mathbb{X}}

\def \Bc{{\mathcal B}}

\def \Smc{{\mathcal S}}
\def \Ac{{\mathcal A}}
\def \Ec{{\mathcal E}}
\def \Pc{{\mathcal P}}
\def \Fc{{\mathcal F}}
\def \Gc{{\mathcal G}}
\def \Uc{{\mathcal U}}

\def \eps{\varepsilon}
\def \Leb{\operatorname{\texttt{Leb}}}

\newcommand{\one}{{\boldsymbol{1}}}

\newcommand{\ebd}{{\boldsymbol{e}}}

\newcommand{\qbar}{{\bar{q}}}\newcommand{\Qbar}{{\bar{Q}}}

\newcommand{\Ybar}{{\bar{Y}}}
\newcommand{\Zbar}{{\bar{Z}}}

\newcommand{\Smchat}{{\hat{\Smc}}}

\newcommand{\Btil}{{\tilde{B}}}
\newcommand{\Ctil}{{\tilde{C}}}

\newcommand{\etatil}{{\tilde{\eta}}}

\newcommand{\Itil}{{\tilde{I}}}

\newcommand{\Ntil}{{\tilde{N}}}

\newcommand{\psitil}{{\tilde{\psi}}}
\newcommand{\qtil}{{\tilde{q}}}

\begin{document}
\title{Moderate Deviation Principle for Join-The-Shortest-Queue-d Systems}
\author{Zhenhua Wang}
	\address{Shandong University}
\author{Ruoyu Wu}
	\address{Iowa State University} 
	\email{zhenhuaw@sdu.edu.cn, ruoyu@iastate.edu}

\date{\today}
\begin{abstract}
    The Join-the-Shortest-Queue-d routing policy is considered for a large system with $n$ servers.
	Moderate deviation principles (MDP) for the occupancy process and the empirical queue length process are established as $n\to \infty$. 
    Each MDP is formulated in terms of a large deviation principle with an appropriate speed function in a suitable infinite-dimensional path space.
	Proofs rely on certain variational representations for exponential functionals of Poisson random measures.
    As a case study, the convergence of rate functions for systems with finite buffer size $K$ to the rate function without buffer is analyzed, as $K \to \infty$.\\


\end{abstract}
\maketitle	

\section{Introduction}\label{sec:introduction}
For $n\in \N$, we consider a large queueing system with $n$ parallel servers, each processing jobs in its queue at rate $1$. Service times and inter-arrival times are exponentially distributed and are mutually independent. Tasks arrive at a single dispatcher at rate $n\lambda$, $\lambda > 0$, and are immediately routed to one of $n$ processors. Upon each task arrival, the dispatcher selects $d$ servers uniformly at random from all $n$ servers and routes the task to the shortest server among the selected $d$ servers. This routing strategy is known as  the join-the-shortest-queue-$d$ (JSQ($d$)) policy, or the power-of-$d$ policy. Let $X^n_k$ be the $k$-th server among all $n$ servers. 

Denote by $\mu^n_i(t):= \frac{\# \{ X^n_k(t)=i\}}{n}$ the proportion of queues with length $i$ at time $t$, and $Q^n_i(t): =\frac{\# \{ X^n_k(t)\geq i \}}{n}= \sum_{j\geq i}\mu^n_j(t)$ the proportion of queues with length at least $i$ at time $t$.  
In this work, we first establish a moderate deviation principle (MDP) for $Q^n$ (see Theorem \ref{thm:MDP}) .
The associated rate function in the above JSQ ($d$) system takes a variational form and is given as the value function of an infinite-dimensional deterministic optimal control problem (see \eqref{eq:rateI}).

Next, we study a second scenario, where a buffer $K\in \N_+$ is set in advance. In this scenario, upon each arrival, the central dispatcher still selects $d$ servers uniformly from all $n$ servers, but only routes the task to the shortest server among the selected $d$ servers if the shortest server has queue length less than $K$. 
Otherwise, the task is discarded. 
We refer to this setup as the JSQ($d$) system with buffer $K$, to distinguish it from the original JSQ($d$) system without buffer constraint.
We provide a MDP with the associated rate function for this buffered system in Theorem \ref{thm:MDP.buffer}. Additionally, we conduct a case study to examine the convergence of the rate function for the JSQ($d$) systems with buffers to the rate function for the original JSQ($d$) system as the buffer size $K\to\infty$.

The JSQ($d$) system considered in this paper is closely related to the traditional join-shortest-queue (JSQ) system, wherein the central dispatcher assigns every task to the shortest server among all $n$ servers at the moment the task arrives. 
Various asymptotic properties of JSQ systems have been explored in the literature, see e.g.\ \cite{BudhirajaFriedlanderWu2019many, banerjee2019join, braverman2020steady, EschenfeldtGamarnik18} and references therein. 
In particular, \cite{EschenfeldtGamarnik18} establishes a central limit theorem under the heavy traffic scaling, and \cite{BudhirajaFriedlanderWu2019many} obtains a large deviation principle (LDP). 

Both JSQ and JSQ($d$) are widely used models for load balancing in distributed resource systems, particularly in applications such as cloud computing, file transfers, and database queries. Implementing the JSQ policy requires instantaneous knowledge of the queue lengths at all servers, which can impose significant communication overhead and may not scale well in environments with a large number of servers. This challenge has led to the exploration of JSQ($d$) systems, where the dispatcher, lacking knowledge of all server lengths, assigns an incoming task to the server with the shortest queue among $d$ servers selected uniformly at random.

The investigation of the JSQ($d$) system within the context of large-scale queueing networks was initiated by  \cite{Mitzenmacher01, VvedenskayaDobrushinKarpelevich96}, and the mean field limit indicates that even a small value of $d=2$ can lead to substantial performance improvements in a many-server regime as $n\to\infty$. 
Since then, this scheme and its various adaptations have been extensively studied, see \cite{aghajani2019hydrodynamic, bramson2012asymptotic,  brightwell2018supermarket, budhiraja2019diffusion,  BudhirajaMukherjeeWu2019supermarket, cardinaels2022power, graham2005functional, luczak2006maximum, luczak2005strong, MukherjeeBorstLeeuwaardenWhiting18, RuttenMukherjee2022load}, and the survey paper \cite{Der2022scalable}. 
In particular, \cite{graham2005functional} provides a central limit theorem for the JSQ (d) system, \cite{luczak2006maximum}  investigates the maximum queue length in the equilibrium distribution, and \cite{luczak2005strong} explores strong approximation results for JSQ($d$) systems.
Additionally, \cite{RuttenMukherjee2022load,BudhirajaMukherjeeWu2019supermarket} analyzes scaling limits of the JSQ(d) system on graphs and its mean-field approximation. 
\cite{brightwell2018supermarket, MukherjeeBorstLeeuwaardenWhiting18} studies the load balancing system in the Halfin?Whitt regime where the number of selected servers is not a fixed integer $d$ but increases with $n$.


The paper is organized as follows. Section \ref{sec:model} introduces the background on the Poisson Random Measures (PRMs) and preliminary results for the JSQ ($d$) system. Proposition \ref{prop:LLN} provides a law of large numbers and describes the deterministic dynamic of the limit system (see \eqref{eq:ODE}). 
In Section \ref{sec:MDP}, Theorem \ref{thm:MDP}  presents the MDP for the JSQ ($d$) system, and Theorem \ref{thm:I} offers an alternative representation for the rate function $I$. The proof for Theorem \ref{thm:MDP} is carried out in Section \ref{sec:proof}. Section \ref{sec:buffer} addresses the MDP result for the JSQ ($d$) system with a buffer.
We then analyze the convergence of the rate function from the JSQ ($d$) systems with buffer $K$ to that in the original JSQ ($d$) system in Section \ref{sec:case}, where the limit dynamic is taken to be the unique stationary solution to \eqref{eq:ODE}.

\section{Model setup and law of large numbers}\label{sec:model}

For a polish space $\S$, denote by $\Bc(\S)$ its Borel $\sigma$-field, and let $\M_b(\S)$ be the space of all real bounded measurable functions on $\S$ with $\|f\|_{L^\infty(\S)}:=\sup_{x\in \S} |f(x)|$ for $f\in \M_b(\S)$.
When $\S$ is equipped with the Lebesgue measure, set $L^2(\S)$ the space of all Borel measurable functions $f$ from $\S$ to $\R$ with $\|f\|_{L^2(\S)}:=\left(\int_{\S} |f(x)|^2 dx\right)^{1/2}<\infty$. 
Denote by $\N_+$ the collection of positive integers, and $\N := \N_+ \cup \{0\}$.

Let $l_2$ be the Hilbert space of square-summable sequences, equipped with the norm $\|\cdot \|_{l^2}$.    Denote by $L(l^2, l^2)$ the space of all bounded linear operators from $l^2$ to itself. We consider a time horizon $[0,T]$ for some $0<T<\infty$. Set $\X:= \R_+$, and $\X_t:=\X\times[0,t]$ for any $t\in(0,T]$, and we adapt several notations from \cite[Section 3]{BudhirajaWu2017moderate} as follows.
\begin{align*}
\Smchat^n:=& \left\{ q=(q_0, q_1, q_2,\dotsc)\in l^2: q_i\geq q_{i+1} \text{ and } nq_i \in \{0,1,\dotsc,n\} \; \forall i, q_0=1 \right\},\\
\Smchat:=& \left\{ q=(q_0, q_1, q_2,\dotsc)\in l^2: q_i\geq q_{i+1} \text{ and }0\leq q_i\leq 1 \; \forall i, q_0=1 \right\},\\
\C([0,T]: l^2):=& \{ \text{continuous functions from $[0,T]$ to $l^2$}\},\\
\D([0,T]: l^2):=& \{ \text{right continuous functions from $[0,T]$ to $l^2$ with left limits}\},\\
\M:=& \{ \text{finite compact measures on $(\X_T, \Bc(\X_T))$}\},
\end{align*}
where $\C([0,T]: l^2)$ and $\D([0,T]: l^2)$ are endowed with the uniform topology and the Skorokhod topology, respectively.

Denote by $X^n_k$ the queue length of $k$-th server among all $n$ servers. Let the arrival rate be $\lambda$ and the service rate be 1. 
Define $\mu^n_i(t):= \frac{\# \{ X^n_k(t)=i\}}{n}$ the proportion of queues with length $i$ at time $t$, and $Q^n_i(t): =\frac{\# \{ X^n_k(t)\geq i \}}{n}= \sum_{j\geq i}\mu^n_j(t)$ the proportion of queues with length at least $i$ at time $t$. 

We will now give a convenient evolution equation for $\{Q^n\}$ in terms of a collection of Poisson random measures. 
For a locally compact metric space $\S$, let $\mathcal{M}_{FC}(\S)$ represent the space of measures $\nu$ on $(\S,\mathcal{B}(\S))$ such that $\nu(K)<\infty$ for every compact $K\in\mathcal{B}(\S)$, equipped with the usual vague topology.
This topology can be metrized such that $\mathcal{M}_{FC}(\S)$ is a Polish space (see \cite{BudhirajaDupuisGanguly2015moderate} for one convenient metric).
A PRM $D$ on $\S$ with mean measure (or intensity measure) $\nu\in\mathcal{M}_{FC}(\S)$ is an $\mathcal{M}_{FC}$-valued random variable such that for each $H\in\mathcal{B}(\S)$ with $\nu(H)<\infty$, $D(H)$ is a Poisson random variable with mean $\nu(H)$ and for disjoint $H_1,\ldots,H_k\in\mathcal{B}(\S)$, the random variables $D(H_1),\ldots, D(H_k)$ are mutually independent random variables (cf. \cite{IkedaWatanabe1990SDE}).

The dynamic of JSQ(d) is given by the following.
\begin{align}
	Q_i^n(t) & = Q_i^n(0) + \frac{1}{n} \int_{[0,1]\times [0,t]} \one_{[0,R_i^n(Q^n(s-)))}(y) \, D_i^{n\lambda}(dy\,ds) \notag \\
	& \quad - \frac{1}{n} \int_{[0,1]\times[0,t]} \one_{[0,Q_i^n(s-)-Q_{i+1}^n(s-))}(y) \, \bar{D}_i^n(dy\,ds), \label{eq:JSQd-Xin} \\
	R_i^n(q) & = \left[\binom{nq_{i-1}}{d} - \binom{nq_i}{d}\right] \Big/ \binom{n}{d}, \quad i \ge 1, \label{eq:Rni}
\end{align}
where $D_i^{n\lambda}(dy\,ds)$ and $\bar{D}_i^n(dy\,ds)$ are mutually independent Poisson random measures (PRMs) on $[0,1]\times [0,T]$ with intensity $n\lambda \, dy\, ds$ and $n\, dy\, ds$, respectively. 
Here $\binom{m}{d}:=\frac{m(m-1)\dotsm(m-d+1)!}{d!}$ is $0$ if $m =0,1,\dotsc,d-1$.

Set $\lambda_0:= \lambda\vee 1$. Denote by $N(dr\, dy\, ds)$ a PRM on $\R_+\times \X_T$ with intensity $dr\, dy\, ds$. Let $N^n$ be a PRM on $\X_T$ defined as  
\bee  
{\color{black}N^n([0,t]\times A):= \int_{\R_+\times A\times [0,t]} 1_{[0, n]}(r) N(dr\, dy\, ds)\quad \forall A\in \Bc(\X),}
\eee 
and define $\Ntil^n(dy\,ds):= N^n(dy\,ds)-n\,dy\,ds$. Then 
define $G^n: \Smchat^n \times \X \to l^2$ as 
\be\label{eq:Gn} 
\begin{aligned}
& G^n(q,y):=\\
& \sum_{i\geq 1}G^n_i(q,y)\ebd_i=\sum_{i\geq 1} \left\{ \one_{[2(i-1)\lambda_0, 2(i-1)\lambda_0+\lambda R^n_i(q))}(y)  - \one_{[(2i-1)\lambda_0, (2i-1)\lambda_0 + (q_i-q_{i+1}))}(y) \right\} \ebd_i,
\end{aligned}
\ee 
where $\ebd_i$ is the $i$-th unit vector in $l^2$.
For each $i$, it follows that the even indexed interval has length equal to $\lambda R^n_i(q)$ while the odd one has length equal to $q_i-q_{i+1}$.
Define $b^n: \Smchat^n \to l^2$ by $b^n(q)=\int_\X G^n(q,y)dy$, that is,
\be\label{eq:bn} 
 b^n(q):=\sum_{i\geq 1} b^n_i(q)\ebd_i:= \sum_{i\geq 1} \left[ \lambda R_i^n(q)-\left( q_i-q_{i+1}\right) \right] \ebd_i.
\ee 
Then \eqref{eq:JSQd-Xin} can be rewritten as 
\begin{align}
	Q^n(t) & = Q^n(0) +\frac{1}{n} \int_{\X_t} G^n(Q^n(s-),y) \, N^{n}(dy\, ds) \notag  \notag \\
	& = Q^n(0) + \int_{[0,t]}b^n(Q^n(s))ds+ \frac{1}{n} \int_{\X_t} G^n(Q^n(s-),y) \, \tilde N^{n}(dy\, ds) \label{eq:JSQd.Qn}.
\end{align}
Intuitively, $G^n$ and $b^n$ above converges to $G: \Smchat \times \X\to l^2$ and $b: \Smchat \to l^2$ respectively, which are defined  as follows.
\be\label{eq:b.G}  
\begin{cases}
\begin{aligned}
G(q,y):=&\sum_{i\geq 1} G_i(q,y)\ebd_i\\
:=& \sum_{i\geq 1} \bigg\{\one_{\big[ 2(i-1)\lambda_0, \,2(i-1)\lambda_0+\lambda(q^d_{i-1}-q^d_i) \big)}(y)
 - \one_{\big[(2i-1)\lambda_0, (2i-1)\lambda_0+(q_{i}-q_{i+1})\big)}(y) \bigg\} \ebd_i,\\
b(q):=&\sum_{i\geq 1} b_i(q)\ebd_i := \sum_{i\geq 1 } \left[ \lambda\left(q_{i-1}^d-q^d_{i} \right)-\left( q_i-q_{i+1}\right)\right] \ebd_i=\int_\X G(q,y)dy.
\end{aligned}
\end{cases}
\ee 



\begin{Lemma}\label{lm:lips}
For each $n\in \N_+$, we have that
\begin{align}
\|b^n(q)-b^n(\bar q) \|_{l^2} &\leq \left(2\lambda dC_n+2\right)\|q-\bar q \|_{l^2}\quad\forall q, \bar q\in \Smchat^n, \label{eq:bn,blips}
\end{align}
where $C_n:=\left( \frac{n}{n-1}\dotsm \frac{n}{n-d+1} \right)$ converges to 1 as $n\to\infty$. Moreover, 
\begin{align}
\|b(q)-b(\bar q) \|_{l^2} & \leq \left(2\lambda d+2\right)\|q-\bar q \|_{l^2}, \quad \forall q, \bar q\in l^2, \label{eq:b,blips}\\
\|b^n(q)-b(q)\|_{l^2} &\leq   \frac{4\lambda d^2}{n}\|q\|_{l^2}\quad \forall q\in \Smchat^n. \label{eq:bn-b.bound} 
\end{align}
\end{Lemma}

\begin{proof}
We first show
\eqref{eq:bn,blips}. 
Given $a\in\{0,\frac{1}{n},\dotsc,1\}$, we can write
\bee
\binom{na}{d} \Big/ \binom{n}{d}=a \left( a-\frac{1}{n} \right) \dotsm \left( a-\frac{d-1}{n} \right)\left( \frac{n}{n-1}\frac{n}{n-2}\dotsm \frac{n}{n-d+1} \right).
\eee 
Now fix $q,\bar q\in \Smchat^n$. 
For each $i \ge 1$, by letting $a=q_i, \bar q_i$ in the above identity respectively and then taking difference, we have that
\bee 
\begin{aligned}
&\left|\binom{nq_i}{d}\Big/ \binom{n}{d}-\binom{n\bar q_i}{d}\Big/ \binom{n}{d}\right|\\
&=\left|  q_{i} \left( q_{i}-\frac{1}{n} \right) \dotsm \left( q_{i}-\frac{d-1}{n} \right)  - \bar q_{i} \left( \bar q_{i}-\frac{1}{n} \right) \dotsm \left( \bar q_{i}-\frac{d-1}{n} \right) \right|  \left( \frac{n}{n-1}\dotsm \frac{n}{n-d+1} \right)\\
&\leq d|q_{i}-\bar q_{i}| \left( \frac{n}{n-1}\dotsm \frac{n}{n-d+1} \right).
\end{aligned}
\eee 
Therefore,
\bee
|R^n_i(q)-R^n_i(\bar q)|\leq d\left( \frac{n}{n-1}\dotsm \frac{n}{n-d+1} \right) \left(|q_{i-1}-\bar q_{i-1}|+|q_{i}-\bar q_{i}| \right).
\eee
Combining this with \eqref{eq:bn}, we have 
\bee
\|b^n(q)-b^n(\bar q) \|_{l^2}\leq \left(2\lambda d\left( \frac{n}{n-1}\dotsm \frac{n}{n-d+1} \right)+2\right)\|q-\bar q \|_{l^2}
\eee
where $\left( \frac{n}{n-1}\dotsm \frac{n}{n-d+1} \right)$ converges to $1$ as $n\to\infty$. 

\eqref{eq:b,blips} follows from a linear expansion on $b$ and we omit the details.

Finally we consider \eqref{eq:bn-b.bound}. 
For all $n\geq 2d$ and $q\in \Smchat^n$, note that
\begin{align*}
    \left| \binom{nq_{i}}{d} \Big/ \binom{n}{d}-q_{i}^d \right| & \le \sum_{k=0}^{d-1} \left| q_{i}^{k}\binom{nq_{i}-k}{d-k} \Big/ \binom{n-k}{d-k}-q_{i}^{k+1}\binom{nq_{i}-(k+1)}{d-(k+1)} \Big/ \binom{n-(k+1)}{d-(k+1)} \right| \\
    & \le \sum_{k=0}^{d-1}q_{i}^k\frac{k(1-q_{i})}{n-k} \leq \frac{2 d^2 q_i}{n}.
\end{align*}
Therefore
\be\label{eq:bn-b}
	|b^n_i(q)-b_i(q)|\leq \lambda \left| \binom{nq_{i-1}}{d} \Big/ \binom{n}{d}-q_{i-1}^d \right|+ \lambda \left| \binom{nq_{i}}{d} \Big/ \binom{n}{d}-q_{i}^d \right| \leq \frac{2\lambda d^2(q_{i-1}+q_i)}{n}.
\ee
This gives \eqref{eq:bn-b.bound}.
\end{proof}


With the help of Lemma \ref{lm:lips}, we have the following law of large numbers result.

\begin{Proposition}[LLN]\label{prop:LLN} 
\bi
\item[\emph{(a)}] For each $n$ and $Q^n(0)\in \Smchat^n$, there exists a measurable mapping $\bar \Gc^n: \M\to \D([0, T]: l^2)$ such that $Q^n=(Q^n(t))_{t\in [0,T]}=\bar\Gc^n(\frac{1}{n} N^n)$.
 
\item[\emph{(b)}] If $\|Q^n(0)-Q(0)\|_{l^2}\to 0$ with $Q^n(0) \in \Smchat^n$ and $Q(0)\in \Smchat$, then $Q^n$ converges in probability to $Q$ in $\D([0,T]: l^2)$, where 
the LLN limit of $Q^n$ is given by a deterministic limit $Q\in \C([0,T]: l^2)$, which is the unique solution to the following infinite dimensional system of ordinary differential equations
\begin{equation}
	\label{eq:ODE}
	Q(t) =Q(0)+\int_{\X_t} G(Q(s),y)dy\,ds=Q(0)+\int_{[0,t]} b(Q(s))ds.
\end{equation}
Moreover, 
\begin{equation}
    \label{eq:Q-bound}
    \| Q(t)\|_{l^2}^2 ds\leq 2\|Q(0)\|_{l^2}^2 e^{8 (\lambda d+1)^2 t}
\end{equation}
\ei 
\end{Proposition}

\begin{proof}
We only prove \eqref{eq:Q-bound}. The rest part is standard and we omit the proof. By the Lipschitz estimate in \eqref{eq:b,blips} and a combination of Minkowski's inequality and Cauchy-Schwarz inequality, we have that
$$
\|Q(t)\|^2_{l^2}\leq 2\|Q(0)\|_{l^2}^2+\int_{[0,t]} 8  (\lambda d+1)^2 \| Q(s)\|_{l^2}^2 ds. 
$$
The result then follows from Gronwall's inequality.
\end{proof}

\section{Moderate deviation principles}\label{sec:MDP}

 To analyze the MDP for the sequence $\{Q^n\}_{n\in\N}$, we begin by applying a Taylor expansion to the function $b: \Smchat \to l^2$ defined in \eqref{eq:b.G}, obtaining that, for any $q,\bar q\in \Smchat$,
 \be\label{eq:bq-bqbar.1} 
 \begin{aligned}
 	(b(\bar q)-b(q))_i=&\lambda d (q_{i-1})^{d-1}(\bar q_{i-1}-q_{i-1})-(\lambda d (q_i)^{d-1}+1)(\bar q_i-q_i)+ (\bar q_{i+1}-q_{i+1})\\
 	&+ \frac{1}{2}\lambda d (d-1) \left(  (\qtil_{i-1})^{d-2} |\bar q_{i-1}-q_{i-1}|^2 - (\qtil_i)^{d-2}|\bar q_{i}-q_{i}|^2 \right)
 \end{aligned}
 \ee 
 with $\qtil_{i-1}, \qtil_i$ between $\qbar_{i-1},q_{i-1}$ and $\qbar_{i},q_{i}$ respectively.
 Then we decompose
 \be\label{eq:Dbthetab} 
 b(\bar q)-b(q)=Db(q)[\bar q-q]+\theta_b(\bar q ,q) \quad \text{with } \theta_b(\bar q ,q):=b(\bar q)-b(q)-Db(q)[\bar q-q],
 \ee 
 where $Db: \Smchat \to L(l^2, l^2)$ represents the linear operation part in the first line of \eqref{eq:bq-bqbar.1}, and is defined by 
 \be\label{eq:Db} 
 Db(q)[p]:=\lambda d T_{-1}\left(q^{\circ d-1} \odot p  \right)- \left(\lambda d q^{\circ d-1}+1\right) \odot p + T_{+1}(p),
 \ee 
 here $(\cdot)^{\circ d-1}$ denotes the entry-wise power of $d-1$, $\odot$ denotes the entry-wise product of two vectors, and $T_{k}$ is the left-shift operator by $k$ entries for $k\in \Z$. 
It is straightforward to verify that for any $\psi\in L^2(\X_T)$, the following ODE has a unique solution in $\C([0,T]: l^2)$
\be\label{eq:Gc}  
{\color{black}\eta(t)=\int_0^t Db(Q(s)) \left[\eta(s) \right] ds+\int_{\X_t} G(Q(s),y)\psi(s,y)dy\, ds.}
\ee 
That is, the mapping 
\be\label{eq:Gc.def}   
\Gc:L^2(\X_t)\to \C([0,T]: l^2)\quad \text{by $\Gc(\psi)=\eta$ if $\eta$ solves \eqref{eq:Gc}},
\ee 
 is well-defined.
 
 Let $\{ a(n)\}_{n\in \N}$ be a sequence such that 
\be\label{eq:an} 
a(n)\sqrt{n}\to\infty \text{ and } a(n)\to 0+.
\ee 
The following theorem provides the MDP for the JSQ(d) system without buffer.
Proofs will be given in Section \ref{sec:proof}.
 
\begin{Theorem}[MDP without buffer]\label{thm:MDP}
Suppose 
\be\label{eq:Qn0-Q0.an}  
a(n)\sqrt{n} \|Q^n(0)-Q(0) \|_{l^2} \to 0
\ee 
with $Q^n(0) \in \Smchat^n$ and $Q(0)\in \Smchat$. Let $Q=(Q(t))_{t\in[0,T]}$ be the unique solution to \eqref{eq:ODE} with initial value $Q(0)\in \Smchat$. 
Then $\{a(n)\sqrt{n}(Q^n-Q)\}_{n}$ satisfies a LDP in $\D([0,T]: l^2)$ with speed $a^2(n)$ and rate function $I:\D([0,T]:l^2)\to [0,\infty]$ defined by 
\be\label{eq:rateI}  
I(\eta):= \inf_{\eta=\Gc(\psi)}\left\{   \frac{1}{2}\|  \psi \|_{L^2(\X_T)}\right\}, \quad \eta \in \C([0,T]:l^2),
\ee 
where
$\Gc$ is defined as in \eqref{eq:Gc.def}, and $I(\eta)=\infty$ if $\eta \in \D([0,T]:l^2) \setminus \C([0,T]:l^2)$.
\end{Theorem}

The following result gives an alternative representation for the rate function $I(\cdot)$ defined in \eqref{eq:rateI}, which is more intrinsic and will be applied in Section \ref{sec:case} for computing explicit formula of rate values in a case study. 
\begin{Theorem}\label{thm:I}
Let $Q=(Q(t))_{t\in[0,T]}$ be the unique solution to \eqref{eq:ODE} with initial value $Q(0)\in \Smchat$. 
For any $\eta\in \D([0,T]: l^2)$ with $I(\eta)<\infty$, we have
\be  
\begin{aligned}
 I(\eta)=&\sum_{j\geq 1} \int_{[0,T]} \one_{\left\{ \lambda(Q^d_{j-1}(t)-Q^d_j(t)) + \left( Q_j(t)-Q_{j+1}(t)\right) > 0 \right\}}\\
 &\qquad \qquad \cdot \frac{\left( \eta'_j(t)- \left[ d\lambda\left( Q_{j-1}^{d-1}(t)\eta_{j-1}(t)-Q_j^{d-1}(t)\eta_{j}(t)\right) - (\eta_j(t)-\eta_{j+1} (t)) \right] \right)^2}{\lambda\left( Q^d_{j-1}(t)-Q^d_j(t) \right)+Q_j(t)-Q_{j+1}(t)} dt.
\end{aligned}
\ee 
\end{Theorem}

\begin{proof}
Suppose $Q=(Q(t))_{t\in[0,T]}$ is the unique solution to \eqref{eq:ODE} with an initial value $Q(0)\in \Smchat$. Fix an arbitrary $\eta\in \D([0,T]: l^2)$ with $I(\eta)<\infty$.
Clearly $\eta'$ exists. 
Now consider $\Gc$ defined in \eqref{eq:Gc.def} and any $\psi$ with $\|\psi\|_{L^2(\X_T)}<\infty$ such that $\eta=\Gc(\psi)$. By the formula of $G(q,y)$ in \eqref{eq:b.G}  and operator $Db$ described in \eqref{eq:bq-bqbar.1} and \eqref{eq:Db}, we have
\begin{align}
\eta'(t) & = Db(Q(t)) \left[\eta(t) \right] +\int_{\X} G(Q(t),y)\psi(t,y)dy \label{eq:eta-prime}\\
& = \sum_{j\geq 1} \left[ d\lambda\left( Q_{j-1}^{d-1}(t)\eta_{j-1}(t)-Q_j^{d-1}(t)\eta_{j}(t)\right) - (\eta_j(t)-\eta_{j+1} (t)) + \phi_j(t)\right] \ebd_j, \notag
\end{align}
where, for each $j\geq 1$,
\begin{align*}
 &\phi_j(t) :=  \int_\X G_j(Q(t), y) \psi(t,y)  \\
 &=\int_{\X}    \bigg\{\one_{\big[ 2(j-1)\lambda_0, \,2(j-1)\lambda_0+\lambda\left( Q^d_{j-1}(t)-Q^d_j(t) \right) \big)}(y)
- \one_{\big[(2j-1)\lambda_0, (2j-1)\lambda_0+ \left( Q_{j}(t)-Q_{j+1}(t) \right)\big)}(y) \bigg\}\psi(t,y)dy.
\end{align*}
For the above arbitrary $\psi$, define $\hat\psi (t,y):= \sum_{j\geq 1}	{\hat\psi}_j (t,y)$ such that for each $j\geq 1$,
\begin{align*}
	{\hat\psi}_j(t,y)
	&:=\one_{\left\{ \lambda(Q^d_{j-1}(t)-Q^d_j(t)) +\left( Q_j(t)-Q_{j+1}(t)\right) > 0 \right\}} \\
	&\quad \cdot \frac{\phi_j(t)\left(\one_{\left\{ \big[2(j-1), 2(j-1)+\lambda(Q^d_{j-1}(t)-Q^d_j(t) \big) \right\}}(y)  - \one_{\left\{\big[2j-1, 2j-1 + Q_j(t)-Q_{j+1}(t) \big)\right\}}(y) \right)}{\lambda\left( Q^d_{j-1}(t)-Q^d_j(t) \right)+Q_j(t)-Q_{j+1}(t)}, t\in[0,T].
\end{align*}
Then 
\[
\int_\X G_j(Q(t),y)\hat \psi_j(t,y) dy = \phi_j(t)\quad \forall j\geq 1, t\in [0,T].
\]
Therefore, $\eta=\Gc(\hat\psi)$ and $I(\eta) \le \frac{1}{2} \|\hat\psi\|_{L^2(\X_T)}$. 
However, by Cauchy-Schwarz inequality, $\|\hat\psi\|_{L^2(\X_T)}\leq \|\psi\|_{L^2(\X_T)}$. 
Therefore, we have $I(\eta) = \frac{1}{2} \|\hat\psi\|_{L^2(\X_T)}$.
Now note that 
\begin{equation*}
    \|\hat\psi\|_{L^2(\X_T)} = \sum_{j \ge 1} \int_{[0,T]} \frac{\one_{\left\{ \lambda(Q^d_{j-1}(t)-Q^d_j(t)) +\left( Q_j(t)-Q_{j+1}(t)\right) > 0 \right\}} \phi_j^2(t)}{\lambda\left( Q^d_{j-1}(t)-Q^d_j(t) \right)+Q_j(t)-Q_{j+1}(t)}dt
\end{equation*}
and from \eqref{eq:eta-prime} we have
\begin{equation*}
    \phi_j(t) = \eta'_j(t)- \left[ d\lambda\left( Q_{j-1}^{d-1}(t)\eta_{j-1}(t)-Q_j^{d-1}(t)\eta_{j}(t)\right) - (\eta_j(t)-\eta_{j+1} (t)) \right].
\end{equation*}
The result follows.
\end{proof}

Using the above MDP for $Q^n$ and a contraction mapping principle with the continuous and injective map $\D([0,T]: l^2) \ni (\nu_i)_{i=0}^\infty \mapsto (\nu_i-\nu_{i+1})_{i=0}^\infty \in \D([0,T]: l^2)$, we immediately get the MDP for $\mu^n$.

\begin{Corollary}
    \label{cor:mu-MDP}
    Under the conditions of Theorem \ref{thm:MDP},
    $\{a(n)\sqrt{n}(\mu^n-\mu)\}_n$ satisfies a LDP with speed $a^2(n)$ and rate function defined by 
    $$\Itil(\etatil) := \begin{cases}
        I(\eta), & \mbox{if } \eta_i = \sum_{j=i}^\infty \etatil_j \mbox{ is such that } \eta \in \C([0,T]:l^2), \\
        \infty, & \mbox{otherwise}.
    \end{cases}$$
    Here $\mu_j = Q_j-Q_{j+1}$ for $j \ge 0$.
\end{Corollary}

\begin{Remark}
    \label{rmk:exp-equiv}
    Without using Theorem \ref{thm:MDP}, one can still obtain Corollary \ref{cor:mu-MDP} via an exponential equivalence argument combined with the MDP result in \cite{BudhirajaWu2017moderate}.
    This short alternative proof will be given in Appendix \ref{sec:exp-equiv}.
    However, since the map $\D([0,T]: l^2) \ni (\nu_i)_{i=0}^\infty \mapsto (\sum_{j=i}^\infty\nu_j)_{i=0}^\infty \in \D([0,T]: l^2)$ is not continuous, Theorem \ref{thm:MDP} cannot be directly obtained by applying the contraction principle to Corollary \ref{cor:mu-MDP}.
\end{Remark}

\section{Proof of Theorem \ref{thm:MDP}}\label{sec:proof}

The basic idea is to verify a sufficient condition for MDP obtained in \cite{BudhirajaDupuisGanguly2015moderate}. 
We begin by adapting a few more notations in \cite[pages 756-757]{BudhirajaWu2017moderate}.
 For $r\in \R_+$, define $\ell: \R_+\to \R_+$ by 
$$
\ell(r):= r\ln (r)-r+1.
$$
Let $\Pc$ be the $\{ \Fc_t\}$-predictable $\sigma$-field on $[0,T]\times\Omega$, and denote by $\Ac_+$ (resp. $\Ac$) the family of all mappings $(\X_T\times\Omega,\Bc(\X) \times \Pc)\to (\R_+, \Bc(\R_+))$ (resp. $(\X_T\times\Omega,\Bc(\X) \times \Pc)\to (\R, \Bc(\R))$). For any $\varphi\in \Ac_+$, define a counting process $N^\varphi$ on $\X_T$ by 
\be\label{eq:def.Nvarphi} 
N^\varphi(A\times [0,t]):= \int_{[0,\infty)\times A\times[0,t]} \one_{[0,\varphi(s,y)]}(r) N(dr\, dy\, ds), \quad \forall t\in [0,T], A\in \Bc(\X).
\ee 
For any Borel measurable $g: \X_T\to \R_+$, define $L_t(g):= \int_{\X_t} \ell(g(s,y)) dy\, ds$, for $t\in[0,T]$. Then $L_t(\varphi)$ is a $[0,\infty]$-valued random variable. 
For each $m\in \N$, define 
\begin{align*}
	\Ac_{b,m}:=\bigg\{  &\varphi\in \Ac: \text{ for all $(\omega,t)\in \Omega\times[0,T]$, $ \frac{1}{m}\leq \varphi(\omega, t, y)\leq m$ if $y\in[0,m]$ }\\
	& \text{ and $\varphi(\omega, t, y) =1$ if $y\notin [0,m]$} \bigg\},
\end{align*}
and let $\Ac_b:= \cup_{m\in \N} \Ac_{b,m}$. Fir each $n\in \N$ and $M\in (0,\infty)$, define the following spaces
\begin{align}
	\Smc_{+, n}^M :=& \left\{   g: \X_T\to \R_+: L_T(g)\leq \frac{M}{a^2(n)n}  \right\}, \notag\\
	\Smc_{n}^M :=& \left\{   f: \X_T\to \R: 1+\frac{1}{a(n)\sqrt{n}} f\in \Smc^M_{+,n} \right\}, \notag\\
	\Uc_{+,n}^M:=& \left\{  \varphi\in \Ac_b: \varphi(\omega, \cdot, \cdot)\in \Smc_{+,n}^M, \P-a.s. \right\}. \notag
\end{align}

Given $M\in(0,\infty)$, denote by $B(M)$ the ball of radius $M$ in $L^2(\X_T)$. A set $\{ \psi^n\}_n\subset \Ac$ with the property that $\sup_{n}\|\psi^n\|_{L^2(\X_T)}\le M$ a.s. for some $M\in (0,\infty)$ will be regarded as a collection of $B(M)$-valued random variables, where $B(M)$ is equipped with the weak topology on the Hilbert space $L_2(\X_T)$. 
Throughout this section, we denote by $B(M)$ the compact metric space obtained by equipping it with the weak topology on $L_2(\X_T)$, and we apply the symbol ``$\Rightarrow$" for weak convergence.

By Lemma \ref{lm:estimate} (b) below, if $g\in \Smc_{+,n}^M$, then with $f:=a(n)\sqrt{n}(g-1)$, $f\one_{\{ |f|\leq a(n)\sqrt{n}\}}\in B\left( \sqrt{M \gamma_2(1)} \right)$, where $\gamma_2(1)\in(0, \infty)$ is defined in Lemma \ref{lm:estimate} (b). 

Proposition \ref{prop:LLN} (a) tells that, for each $n$, there exists a measurable mapping $ \Gc^n: \M\to \D([0, T]: l^2)$ such that $a(n)\sqrt{n}(Q^n-Q)=\Gc^n(\frac{1}{n} N^n)$. Also recall the mapping $\Gc:L^2(\X_T)\to \C([0,T]:l^2)$ defined in \eqref{eq:Gc.def}. Then by \cite[Theorem 2.3]{BudhirajaDupuisGanguly2015moderate}, 
the MDP in Theorem \ref{thm:MDP} holds if the following Condition \ref{con:MDP} on $\{\Gc^n\}_{n\in\N}$ and $\Gc$ is satisfied.
\begin{Condition}\label{con:MDP}
\emph{(a)} Given $M\in (0,\infty)$, suppose that $g^n, g\in B(M)$ and $g^n\to g$. Then
$$
\Gc(g^n)\to \Gc(g)\text{ in $\D([0,T]:l^2)$ }. 
$$
\emph{(b)} Given $M\in (0,\infty)$, let $(\varphi^n)_{n}$ be a sequence such that, for each $n$, $\varphi^n\in \Uc^M_{+,n}$ 
and for some $\beta\in(0,1]$, 
$$
\psi^n \one_{\{|\psi^n |\leq \beta a(n)\sqrt{n}\}} \Rightarrow \psi \text{ in $B\left(\sqrt{M \gamma_2(1)} \right)$ for some $\psi$, where $\psi^n:= a(n)\sqrt{n}(\varphi^n-1)$}.
$$
Then
$
\Gc^n\left(  \frac{1}{n} N^{n\varphi^n}\right) \Rightarrow \Gc(\psi).
$
\end{Condition}

In the rest part of this section, we shall verify (a) and (b) in Condition \ref{con:MDP} separately. 
We shall follow the idea carried in \cite[Section 5]{BudhirajaWu2017moderate}. However, non-trivial efforts have to be made toward the sequence of functionals $G^n$ and $b^n$ which are specific for our model. 
Recall $Q(t)$ is the unique solution to \eqref{eq:ODE} with initial value $Q(0)$.

\subsection{Verification of Condition \ref{con:MDP} (a)}\label{sec:proof.a}
Recall \eqref{eq:Gn} and \eqref{eq:b.G}. We have the following estimates.

\begin{Lemma}\label{lm:1}
For all $k\in \N$, 
$$\sup_{n\geq d} \int_\X \| G^n(q,y)\|_{l^2}^k dy\leq \lambda \left( \frac{d^2+d}{2} \right)+1 , \forall q\in \Smchat^n, \quad \int_\X \| G(q,y)\|_{l^2}^k dy \leq \lambda +1, \forall q\in \Smchat.$$
\end{Lemma}

\begin{proof}
Given $q\in \Smchat^n$, let $K:=\sup\{i\in \N: nq_i\geq d\}\in \N\cup\{\infty\}$. Then
\begin{align*}
&0\leq R^n_i(q)=\prod_{k=0}^{d-1} \frac{nq_{i-1}-k}{n-k}-\prod_{k=0}^{d-1} \frac{nq_{i}-k}{n-k}\leq \sum_{k=0}^{d-1}\frac{n(q_{i-1}-q_i)}{n-k}\\
&\qquad \qquad \; \leq \left( \sum_{k=0}^{d-1}\frac{n}{n-k} \right) (q_{i-1}-q_i)\leq\left( \sum_{k=1}^{d} k \right) (q_{i-1}-q_i)\leq  \frac{d^2+d}{2}(q_{i-1}-q_i)\quad \forall i\leq K,\\
&0\leq R^n_{K+1}(q) = \prod_{k=0}^{d-1} \frac{nq_{K}-k}{n-k}\leq q_K,\\
&0\leq R^n_{i+1}(q) = 0\quad \forall i>K.
\end{align*}
Then
\begin{align*}
&\int_\X \| G^n(q,y)\|_{l^2}^k dy\\
&=\int_\X \left( \sum_{i\geq 1}\left(  \one_{[2(i-1)\lambda_0, 2(i-1)\lambda_0+\lambda R_i^n(q) )}(y)  + \one_{[(2i-1)\lambda_0, (2i-1)\lambda_0+(q_{i}-q_{i+1}))}(y)  \right) \right) dy\\
&=\sum_{i\geq 1} [\lambda R^n_i(q)+(q_i-q_{i+1})]
\leq  \lambda  \left( \frac{d^2+d}{2} \right)  \sum_{1\leq i\leq K} (q_{i-1}-q_i) + \lambda q_K + \sum_{i\geq1} (q_i-q_{i+1})\\
&\le \lambda \left( \frac{d^2+d}{2} \right)+1.
 \end{align*}
 A similar argument gives the upper bound for $G$.
\end{proof}


\begin{Lemma}\label{lm:3}
Fix $M\in (0,\infty)$. 
Suppose $g^n, g\in B(M)$ and $g^n\to g$. Then
\begin{align*}
\int_{[0,\cdot]\times \X} g^n(s,y)G(Q(s),y)dy\, ds \to \int_{[0,\cdot]\times \X} g(s,y)G(Q(s),y)dy\, ds\quad \text{in }\C([0,T]: l^2).
\end{align*}
\end{Lemma}

\begin{proof}
By Lemma \ref{lm:1}, $(s,y) \mapsto G_i(Q(s),y)$ is in $L^2(\X_T)$ for each $i \ge 1$.
Since $g^n \to g$ in $B(M)$, we have
$$\int_{[0,t]\times \X} g^n(s,y)G_i(Q(s),y)dy\, ds \to \int_{[0,t]\times \X} g(s,y)G_i(Q(s),y)dy\, ds$$
for each $t \in [0,T]$ and $i \ge 1$.
By Cauchy-Schwarz inequality,
$$
\left| \int_{\X_t} g^n(s,y) G_i(Q(s), y) dy\, ds \right|\leq M \left( \int_{\X_t} \left(G_i(Q(s), y) \right)^2 dy\, ds \right)^{1/2}
$$
and the right hand side is square-summable by Lemma \ref{lm:1}.
It then follows from Dominated Convergence Theorem that
\begin{equation}
    \label{eq:gn-g-t}
    \int_{[0,t]\times \X} g^n(s,y)G(Q(s),y)dy\, ds \to \int_{[0,t]\times \X} g(s,y)G(Q(s),y)dy\, ds, \quad \forall t \in [0,T].
\end{equation}
Moreover, by Cauchy-Schwarz inequality and Lemma \ref{lm:1}, we have 
\begin{align*}
& \left\|   \int_{\X \times [s,t]} g^n(r,y)G(Q(r),y)dy\, dr \right\|_{l^2}^2\leq  M^2   \int_{\X \times [s,t]}\left\| G(Q(r),y)\right\|_{l^2}^2dy\, dr\\
& \leq  \left( \lambda+1\right)M^2 |t-s|,\quad \forall\, 0\leq s\leq t\leq T.
\end{align*}
This gives equicontinuity and hence implies the uniform convergence of \eqref{eq:gn-g-t} in $t \in [0,T]$.
\end{proof}

Now we verify Condition \ref{con:MDP} (a).

\begin{Proposition}
    Given $M\in (0,\infty)$, suppose that $g^n, g\in B(M)$ and $g^n\to g$. Then
    $$
    \Gc(g^n)\to \Gc(g)\text{ in $\D([0,T]:l^2)$ }. 
    $$
\end{Proposition}

\begin{proof}
    By \eqref{eq:Gc} and \eqref{eq:Gc.def},
    $$
    \Gc(g^n)(t)-\Gc(g)(t) = \int_{[0,t]}Db(Q(s)) \left[ \Gc(g^n)(s)-\Gc(g)(s)\right]ds +\int_{\X_t}\left[ g^n(s,y)- g(s,y)\right]G(s,y)dy\,ds.
    $$
    The result then follows from Lemma \ref{lm:3}, Lemma \ref{lm:estimate} (a) and the Gronwall's inequality.
\end{proof}

\subsection{Verification of Condition \ref{con:MDP} (b)}\label{sec:proof.b}

Recall $b^n, b, Db, \theta_b$ defined in \eqref{eq:bn}, \eqref{eq:b.G}, \eqref{eq:Dbthetab} and \eqref{eq:Db}. To begin with, we provide the following estimates.

\begin{Lemma}\label{lm:estimate}
\emph{(a)}
\begin{align}
&\|Db(q)\|_{L(l^2,l^2)}\leq 2\lambda d+2,\quad  \|\theta_b(\qbar,q)\|_{l^2} \leq \lambda d(d-1) \|\qbar - q\|_{l^2}^2, \quad \forall q,\qbar\in \Smchat. \label{eq:Dbthetab.estimate}  
\end{align}
	\emph{(b)} Suppose $g\in \Smc_{+, n}^M$ for some $M\in (0,\infty)$. Let $f:= a(n)\sqrt{n}(g-1)\in \Smc_n^M$. Then
\begin{align}
	&\int_{\X_T} |f| \one_{\{ |f|\geq \beta a(n)\sqrt{n}\}} dy\, ds \leq \frac{M \gamma_1(\beta)}{a(n)\sqrt{n}}\quad \quad \forall \beta >0, \label{eq:lmold.1}\\
	&\int_{\X_T} |g| \one_{\{ |g|\geq \beta \}} dy\, ds \leq \frac{M \gamma_1'(\beta)}{a^2(n)n}\quad \quad \qquad \; \; \; \forall \beta >1, \label{eq:lmold.2}\\
	&\int_{\X_T} |f|^2 \one_{\{ |f|\leq \beta a(n)\sqrt{n}\}} dy\, ds \leq M \gamma_2(\beta)\quad \quad \forall \beta >0, \label{eq:lmold.3}
\end{align}
where $\gamma_1(\cdot), \gamma_1'(\cdot), \gamma_2(\cdot): (0,\infty)\to (0,\infty)$ such that
\begin{align*}
	|x-1|\leq & \gamma_1(\beta)\ell(x)\; \text{ for }|x-1|\geq \beta \text{ and }x\geq 0;\\
	x\leq & \gamma_1'(\beta) \ell(x)\; \text{ for } x\geq \beta>1;\\
	|x-1|^2 \leq &\gamma_2(\beta) \ell(x)\; \text{ for } |x-1|\leq \beta \text{ and }x\geq 0.
\end{align*}
Furthermore, $\gamma_1$ can be chosen such that $\gamma_1(\beta)\leq \frac{4}{\beta}$ for any $\beta\in (0,1/2)$. 
\end{Lemma}

\begin{proof}
Part (b) is a summary of \cite[Lemmas 3.1, 3.2]{BudhirajaDupuisGanguly2015moderate}. As for part (a),
take any $q\in \Smchat$,
	\begin{align*}
		&\left\| \lambda d T_{-1}\left(\left(q^{\circ d-1}\right) \odot p  \right) \right\|_{l^2}\leq \lambda d \|q^{\circ d-1}\|_{l^\infty} \|p\|_{l^2}\leq \lambda d \|p\|_{l^2},\\
		&\left\| \left(\lambda d  q^{\circ d-1}+1\right) \odot p \right\|_{l^2} \leq (\lambda d+1)\|p\|_{l^2},\\
		&\|  T_{+1}(p) \|_{l^2}\leq \|p\|_{l^2}.
	\end{align*}  
	Then the first inequality in \eqref{eq:Dbthetab.estimate} follows. The second inequality in \eqref{eq:Dbthetab.estimate} follows from \eqref{eq:bq-bqbar.1}.
\end{proof}

\begin{Lemma}\label{lm:SnM+prop}
Let $M\in (0,\infty)$. There exists $\gamma_3\in (0,\infty)$ such that for all measurable maps $q: [0,T]\to \Smchat^n$,
$$
\sup_{n\in \N} \sup_{g\in \Smc_{+, n}^M} \int_{\X_T} \|G^n(q(s),y)\|_{l^2}^2 g(s,y)dy\, ds \leq \gamma_3.
$$
\end{Lemma}

\begin{proof}
Fix an arbitrary $g\in \Smc_{+,n}^M$. 
By \eqref{eq:lmold.2} in Lemma \ref{lm:estimate} (b) and Lemma \ref{lm:1}, we have that 
\begin{align*}
\int_{\X_T} \|G^n(q(s),y)\|_{l^2}^2 g(s,y)dy\, ds
&\leq (\lambda+1) \int_{g\geq 2} g(s,y) dy\, ds +2 \int_{\X_T} \| G^n(q(s), y)\|_{l^2}^2 dy\, ds \\ &\leq \frac{(\lambda+1) M \gamma_1' (2)}{ a^2(n)n} + 2\left( \lambda\left(\frac{d^2+d}{2} \right) +1\right) T.
\end{align*}
The result follows on noting that $a^2(n)n \to \infty$ as $n \to \infty$.
\end{proof}

\begin{Lemma}\label{lm:SmcnM.bound}
Take $M\in (0,\infty)$. There exists a mapping $\gamma_4: (0,\infty)\to (0,\infty)$ such that for all $n\geq d$, $\beta\in (0,\infty)$, measurable subset $D\subset [0,T]$ and measurable mapping $q: [0,T]\to \Smchat^n$,
\begin{align}
\sup_{f\in \Smc_n^M} \int_{D\times \X} \|G^n(q(s), y )\|_{l^2} |f(s,y)| \one_{\{ |f(s,y)| \geq  \beta a(n)\sqrt{n} \}} dy\,ds\leq \frac{\gamma_4(\beta)}{a(n)\sqrt{n}} \label{eq:lm5.0}.
\end{align}
Moreover, there exists $\gamma_5\in(0,\infty)$ such that for all $n\geq d$, measurable subset $D\subset [0,T]$ and measurable mapping $q: [0,T]\to \Smchat^n$,
\begin{align}
	\sup_{f\in \Smc_n^M} \left\|\int_{D\times \X} G^n(q(s), y )|f(s,y)| dy\,ds \right\|_{l^2} \leq  \gamma_5\left(\frac{1}{a(n)\sqrt{n}}+\sqrt{ \Leb(D) } \right) \label{eq:lm5.1},
\end{align}
where $ \Leb(D)$ denotes the Lebesgue measure of $D$. 
\end{Lemma}

\begin{proof}
Fix an arbitrary $f\in \Smc_n^M$. 
Notice that $\| G^n(q,y)\|_{l^2}\leq 1$. It follows from \eqref{eq:lmold.1} in Lemma \ref{lm:estimate} (b) that
\begin{align}
&\int_{D\times \X} \| G^n(q(s),y)\|_{l^2}|f(s,y)| \one_{ \{ |f(s,y)| \geq \beta a(n)\sqrt{n} \}} dy\,ds \notag\\
& \leq  \int_{D\times \X} |f(s,y)| \one_{\{ |f(s,y)| \geq \beta a(n)\sqrt{n} \}} dy\,ds\leq \frac{M\gamma_1(\beta)}{a(n)\sqrt{n}}, \label{eq.lm5:3}
\end{align} 
which tells \eqref{eq:lm5.0}.
By Cauchy-Schwarz inequality, we have that 
\begin{align}
	&\int_{D\times \X} \| G^n(q(s),y)\|_{l^2}|f(s,y)| \one_{\{ |f(s,y)| < \beta a(n)\sqrt{n} \} } dy\,ds \notag\\
	& \leq \left(\int_{D\times \X} \| G^n(q(s),y)\|_{l^2}^2 dy\,ds\right)^{1/2} \left( \int_{D\times \X} |f(s,y)|^2 \one_{ \{ |f(s,y)| < \beta a(n)\sqrt{n}  \} } dy\,ds\right)^{1/2} \notag\\
	&\leq \sqrt{ \left( \lambda \left( \frac{d^2+d}{2} \right)+1\right) M\gamma_2(\beta)|D|}, \label{eq:lm5.5}
	\end{align}
	where the last inequality follows from \eqref{eq:lmold.3} in Lemma \ref{lm:estimate} (b) and Lemma \ref{lm:1}. As a consequence, \eqref{eq:lm5.1} follows from a combination of \eqref{eq.lm5:3}, \eqref{eq:lm5.5} and taking $\beta = 1$.
\end{proof}

For $\varphi\in\Uc_{+,n}^M$, define $\Zbar^{n\varphi}:= \Gc^n(\frac{1}{n} N^{n\varphi})$, where $N^{n\varphi}$ is defined as in \eqref{eq:def.Nvarphi}. Then it follows from an application of Girsanov's theorem (see e.g., the arguments above Lemma 4.6 in \cite{BudhirajaDupuisGanguly2015moderate}) that
\be\label{eq:Znvarphi}  
\Zbar^{n, \varphi}=a(n)\sqrt{n}(\Qbar^{n,\varphi}-Q),
\ee 
where $\Qbar^{n,\varphi}$ is the unique 
solution of 
$$
\Qbar^{n,\varphi}(t) = Q^n(0)+\frac{1}{n} \int_{\X_t} G^n(\Qbar^{n,\varphi}(s-), y) N^{n\varphi}(dy\, ds).
$$
Define $\|\cdot\|_{*, t}:=\sup_{0\leq s\leq t} \|  \cdot \|_{l^2}$. We have the following moment bound for $\Zbar^{n, \varphi}$. 

\begin{Lemma}\label{lm:Zbar.estimate}
$$\sup_{n\in \N} \sup_{\varphi\in \Uc_{+,n}^M} \E\left\|  \Zbar^{n,\varphi} \right\|^2_{*, T}<\infty$$
\end{Lemma}

\begin{proof}
For any $\varphi \in \Uc_{+,n}^M$, let $\Ntil^{n\varphi}(dy\, ds):= N^{n\varphi}(dy\, ds)-n\varphi(s,y)dy\, ds$ and $\psi:= a(n)\sqrt{n}(\varphi-1)$.
Recall \eqref{eq:Gn}, \eqref{eq:b.G} and  \eqref{eq:ODE}.
\bee 
\begin{aligned}
\Qbar^{n,\varphi}(t)-Q(t)=&Q^n(0)-Q(0) +\frac{1}{n}\int_{\X_t} G^n(\Qbar^{n,\varphi}(s-),y)N^{n\varphi}(dy\,ds)-\int_{\X_t} G(Q(s),y)dy\, ds\\
=&Q^n(0)-Q(0)+\frac{1}{n} \int_{\X_t} G^n(\Qbar^{n,\varphi}(s-),y)\Ntil^{n\varphi}(dy\, ds)\\
&+\int_{\X_t} G^n(\Qbar^{n,\varphi}(s),y)(\varphi(s,y)-1)dy\, ds\\
&+ \int_{\X_t} \left[ G^n(\Qbar^{n,\varphi}(s),y)- G(Q(s),y) \right]dy\, ds.
\end{aligned}
\eee 
Above together with \eqref{eq:Znvarphi} tells that 
\be\label{eq: Z.decomp}  
\Zbar^{n,\varphi}= A^n+M^{n,\varphi} +B^{n,\varphi} +C^{n,\varphi},
\ee 
where 
\begin{align*}
A^n:= & a(n)\sqrt{n} (Q^n(0)-Q(0)),\\
M^{n,\varphi}(t):=& \frac{a(n)}{\sqrt{n}}  \int_{\X_t} G^n(\Qbar^{n,\varphi}(s-),y)\Ntil^{n\varphi}(dy\, ds)\\
B^{n,\varphi}(t):= & a(n)\sqrt{n} \int_{\X_t} \left[ G^n(\Qbar^{n,\varphi}(s),y)- G(Q(s),y) \right]dy\, ds\\
=& a(n)\sqrt{n} \int_{[0,t]} \left[  b^n(\Qbar^{n,\varphi}(s)) - b(Q(s)) \right] ds, \\
 C^{n,\varphi}(t) :=& \int_{\X_t} G^n(\Qbar^{n,\varphi}(s),y)\psi(s,y)dy\, ds.
\end{align*}
For each $n\in \N$, $M^{n,\varphi}$ is a martingale. 
Then Doob's inequality tells that 
$$
\E \| M^{n,\varphi}\|^2_{*, T}\leq \frac{4 a^2(n)}{n} \E\left[  \int_{\X_T} \| G^n(\Qbar^{n,\varphi}(s), y) \|_{l^2}^2 n\varphi(s,y)dy\, ds  \right]. 
$$
By $\varphi\in \Uc_{+,n}^M$, it follows from Lemma \ref{lm:SnM+prop} and above inequality that 
\be\label{eq:Mn} 
\sup_{\varphi\in \Uc_{+,n}^M} \E \left\| M^{n,\varphi} \right\|^2_{*,T}\leq \kappa_1 a^2(n) \quad \text{for a finite constant $\kappa_1$}.
\ee 
Notice that $\psi\in \Smc_{n}^M$ a.s., then \eqref{eq:lm5.1} in Lemma \ref{lm:SmcnM.bound} tells that 
\be\label{eq:Cn}  
\sup_{\varphi\in \Uc_{+,n}^M} \| C^{n,\varphi}\|^2_{*, T}\leq \kappa_2 \left( \frac{1}{a^2(n)n}+T\right)\quad \text{for a finite constant $\kappa_2$}.
\ee 
By \eqref{eq:bn,blips}, for all $n$ big enough,
\be\label{eq:Bnvarphi1} 
\begin{aligned}
& a(n)^2n \int_{[0,t]} \left\| b^n(\Qbar^{n,\varphi}(s)) - b^n(Q(s)) \right\|^2_{l^2} ds\\
&\leq   a^2(n)n \int_{[0,t]} 4(\lambda dC_n+1)^2  \| \Qbar^{n,\varphi}(s)-Q(s) \|^2_{l^2} ds\\
&\leq   4(\lambda d2+1)^2  \int_{[0,t]}  \| \Zbar^{n,\varphi} \|^2_{*,s} ds.
\end{aligned}
\ee 
For each $i\in \N$, \eqref{eq:bn-b.bound} in Lemma \ref{lm:lips} and Proposition \ref{prop:LLN} imply that 
\be\label{eq:Bnvarphi3} 
\begin{aligned}
& a^2(n)n\int_{[0,t]} \left\| b^n(Q(s)) - b(Q(s)) \right\|^2_{l^2} ds\\
& \leq \frac{a^2(n)}{n}16\lambda^2d^4 \int_{[0,t]} \left\| Q(s) \right\|_{l^2}^2 ds \leq \frac{a^2(n)}{n}16\lambda^2 d^4  2\|Q(0)\|_{l^2}^2 e^{8 (\lambda d+1)^2 T} T,\quad \forall 0\leq t\leq T.
\end{aligned}
\ee 
Then combining \eqref{eq:Bnvarphi1} with \eqref{eq:Bnvarphi3}, 
\be\label{eq:Bnvarphi5} 
\begin{aligned}
& \|B^{n,\varphi}\|_{*, t}^2\leq a^2(n)n T\int_{[0,t]} \left\| b^n(\Qbar^{n,\varphi}(s)) - b(Q(s))  \right\|^2_{l^2} ds \\
& \leq a^2(n)n T\int_{[0,t]} \left\| b^n(\Qbar^{n,\varphi}(s)) - b^n(Q(s))  \right\|^2_{l^2} ds+ a^2(n)n T\int_{[0,t]} \left\| b^n(Q(s)) - b(Q(s))  \right\|^2_{l^2} ds\\
&\leq 4(2\lambda d+1)^2T  \int_{[0,t]}  \| \Zbar^{n,\varphi} \|^2_{*,s} ds+16\lambda^2d^4 T \frac{a^2(n)}{n} 2\|Q(0)\|_{l^2}^2 e^{8 (\lambda d+1)^2 T}T,
\end{aligned}
\ee  
where the first inequality follows from Cauchy-Schwarz inequality and the last follows from Lemma \ref{lm:lips}. 
By collecting estimates in \eqref{eq:Mn}, \eqref{eq:Cn} and \eqref{eq:Bnvarphi5} and combining with \eqref{eq:Znvarphi}, we have that 
$$
\E\|\Zbar^{n,\varphi} \|_{*,t}^2\leq \kappa_3\left(  \|A^n\|_{l^2}^2+1+a^2(n)+\frac{1+a^4(n)}{a^2(n)n} +\int_{[0,t]} \E\left\| \Zbar^{n,\varphi} \right\|_{*,s}^2 ds \right),
$$
where $\kappa_3$ is a finite constant. Then the desired result follows from Gronwall's inequality and \eqref{eq:Qn0-Q0.an}.
\end{proof}

\begin{Lemma}\label{lm:Gnqbar-Gq}
There exists $\gamma_5\in (0,\infty)$ such that for all $n\in\N$ , $g\in \M_b(\X)$ and $q \in \Smchat,\qbar\in \Smchat^n$,
$$
\left\| \int_\X \left( G^n(\qbar, y)-G(q,y)  \right) g(y) dy \right\|_{l^2}  \leq \frac{4 \lambda d^2}{n}  \|g\|_{L^\infty(\X)}  \|q\|_{l^2}+\gamma_5(1+\frac{1}{n})\|g\|_{L^\infty(\X)} \|\qbar-q \|_{l^2}
$$
\end{Lemma}

\begin{proof}
\begin{align*}
\left\| \int_\X \left( G^n(\qbar, y)-G(q,y)  \right) g(y) dy \right\|_{l^2} \leq & \left\| \int_\X \left( G^n(\qbar, y)-G(\qbar,y)  \right) g(y) dy \right\|_{l^2} \\
&+ \left\| \int_\X \left( G(\qbar, y)-G(q,y)  \right) g(y) dy \right\|_{l^2}.
\end{align*}
\eqref{eq:bn-b.bound} and \eqref{eq:b,blips} in Lemma \ref{lm:lips} indicate that the first term and the second term on the right-hand side above are bounded by $\frac{4\lambda d^2}{n}  \|g\|_{L^\infty(\X)}  \|\qbar\|_{l^2}$ and $\|g\|_{L^\infty(\X)}(2\lambda d +2) \|q-\qbar \|_{l^2}$ respectively, and the desired result follows.
\end{proof}

\begin{Lemma}\label{lm:h}
Let $M\in(0,\infty)$. Let $\eps\in \D([0,T] : l^2)$ and $f\in B(M)$, the following has a unique solution $z\in  \D([0,T] : l^2)$ 
$$
z(t)=\eps(t)+ \int_{[0,t]} Db(Q(s)) [z(s)]ds+\int_{\X_t} G(Q(s), y) f(s,y)dy\, ds, \quad t\in [0,T].
$$
That is, there exists a measurable map $h:  \D([0,T] : l^2)\times B(M)\to  \D([0,T] : l^2)$ such that the solution can be written as $z=h(\eps, f)$. Moreover, $h$ is continuous at $(0,f)$ for every $f\in B(M)$.
\end{Lemma}

\begin{proof}
The existence, uniqueness and the measurability of the solution, as discussed in the proof of \cite[Lemma 5.10]{BudhirajaWu2017moderate}, is easy to check under the bound for $Db(q)$ provided in Lemma \ref{lm:estimate} (a). The continuity of $h$ at $0, f$ follows a similar argument as in the proof of \cite[Lemma 5.10]{BudhirajaWu2017moderate} by applying Lemma \ref{lm:3}.
\end{proof}

Now we verify Condition \ref{con:MDP} (b). 

\begin{Proposition}
Fix $M\in (0,\infty)$. Consider $(\varphi^n)_{n\in \N}$ such that $\varphi^n\in \Uc_{+,n}^M$ for each $n\in \N$. Let $\psi^n:= a(n)\sqrt{n} (\varphi^n-1)$ such that $\psi^n\one_{\{|\psi^n |\leq \beta a(n)\sqrt{n}\}}\Rightarrow \psi$ in $B(\sqrt{M \gamma_2(1)})$ for some $\beta\in (0,1]$. Then $\Gc^n(\frac{1}{n}N^{n\varphi^n}) \Rightarrow \Gc(\psi)$.
\end{Proposition}

\begin{proof}
We will follow the notations in \eqref{eq:Znvarphi} and the formulas after that in \eqref{eq: Z.decomp}. We have that 
\be  
\Gc^n\left( \frac{1}{n}N^{n\varphi^n} \right)=\Zbar^{n\varphi^n}=A^n+M^{n,\varphi^n}+B^{n,\varphi^n}+C^{n,\varphi^n}.
\ee 
By \eqref{eq:Qn0-Q0.an} and \eqref{eq:Mn}, we have both $\|A^n\|_{l^2}^2\to 0 $ and $\E\left[ \left\| M^{n,\varphi^n} \right\|^2_{*,T}  \right] \to 0$  as $n\to \infty$.

As for term $B^{n,\varphi^n}$, by \eqref{eq:Dbthetab}
\be\label{eq:prop.Bn}  
\begin{aligned}
B^{n,\varphi^n}(t)=&a(n)\sqrt{n} \int_{[0,t]} \left[  b^n\left( \Qbar^{n,\varphi^n}(s) \right) - b\left( \Qbar^{n,\varphi^n}(s) \right) \right] ds\\
&+a(n)\sqrt{n} \int_{[0,t]} Db(Q(s))\left[ \Qbar^{n,\varphi^n}(s)- Q(s) \right] ds\\
&+ a(n)\sqrt{n} \int_{[0,t]} \theta_b\left( \Qbar^{n,\varphi^n}(s), Q(s) \right)ds\\
=& \Btil^{n,\varphi^n}(t)+\Ec_1^{n,\varphi^n}(t),
\end{aligned}
\ee 
where $ \Btil^{n,\varphi^n}(t):= \int_{[0,t]} Db(Q(s))\left[ \Zbar^{n,\varphi^n}(s) \right] ds $ and 
\begin{align*}
\Ec_1^{n,\varphi^n}(t):=a(n)\sqrt{n} \int_{[0,t]} \left[  b^n\left( \Qbar^{n,\varphi^n}(s) \right) - b\left( \Qbar^{n,\varphi^n}(s) \right) \right] ds+a(n)\sqrt{n} \int_{[0,t]} \theta_b\left( \Qbar^{n,\varphi^n}(s), Q(s) \right)ds.
\end{align*}
Combining Minkowski's inequality and Cauchy-Schwarz inequality, then applying \eqref{eq:bn-b.bound} in Lemma \ref{lm:lips} and \eqref{eq:Dbthetab.estimate} in Lemma \ref{lm:estimate} (a) gives that
\be\label{eq:prop.Ec1} 
\begin{aligned}
\E\left\|  \Ec_1^{n,\varphi^n} (t)\right\|_{*, T}\leq & a(n)\sqrt{n} \frac{4\lambda d^2}{n}T\left(  \E\int_{[0,T]}  \|\Qbar^{n,\varphi^n}(s)-Q(s)\|_{l^2}ds+ \int_{[0,T]}  \|Q(s)\|_{l^2}ds \right)\\
&+a(n)\sqrt{n} 2\lambda d(d-1) T \left(  \E\int_{[0,T]}  \|\Qbar^{n,\varphi^n}(s)-Q(s)\|_{l^2}^2ds \right)\\
\leq & \frac{4\lambda d^2}{n}T  \left(\E \|\Zbar^{n,\varphi^n}\|_{*,T}+\int_{[0,T]}  \|Q(s)\|_{l^2}ds \right)+\frac{2\lambda d(d-1)}{a(n)\sqrt{n}} T \E\|\Zbar^{n,\varphi^n}\|_{*,T}^2\\
\to& 0, \text{as  } n\to\infty,
\end{aligned}
\ee 
where the last line follows from the boundedness result in Lemma \ref{lm:Zbar.estimate}.

As for term $C^{n,\varphi^n}$, we write $C^{n,\varphi^n}=\Ctil^{n,\varphi^n}+\Ec_2^{n,\varphi^n}+\Ec_3^{n,\varphi^n}+\Ec_4^{n,\varphi^n}$, where
\begin{align*}
\Ctil^{n,\varphi^n}:= & \int_{\X_t} G(Q(s), y) \psi^n(s,y) \one_{\{|\psi^n |\leq \beta a(n)\sqrt{n}\}} dy\, ds,\\
\Ec_2^{n,\varphi^n}:= & \int_{\X_t} G(Q(s), y) \psi^n(s,y) \one_{\{|\psi^n |> \beta a(n)\sqrt{n}\}} dy\, ds,\\
\Ec_3^{n,\varphi^n}:= & \int_{\X_t} \left(G^n\left( \Qbar^{n,\varphi^n}(s),y\right) -G(Q(s), y) \right) \psi^n(s,y) \one_{\{|\psi^n |>(a(n)\sqrt{n})^{1/2} \}} dy\, ds,\\
\Ec_4^{n,\varphi^n}:= & \int_{\X_t} \left(G^n\left( \Qbar^{n,\varphi^n}(s),y\right) -G(Q(s), y) \right) \psi^n(s,y) \one_{\{|\psi^n |\leq  (a(n)\sqrt{n})^{1/2}  \}} dy\, ds.
\end{align*}
By \eqref{eq:lm5.0} in Lemma \ref{lm:SmcnM.bound} (notice that the estimate is also valid for $G$) , we have that 
\be\label{eq:prop.Ec2}  
\left\|  \Ec_2^{n,\varphi^n} \right\|_{*,T}\leq \frac{\gamma_4(\beta)}{a(n)\sqrt{n}} \to 0 \text{ as } n\to\infty.
\ee 
By taking $\beta=(a(n)\sqrt{n})^{-1/2}$ in \eqref{eq:lmold.1},
we have that 
\be\label{eq:prop.Ec3}  
\begin{aligned}
\left\|  \Ec_3^{n,\varphi^n} \right\|_{*,T}\leq  & \int_{\X_T} \left\|  G^n\left( \Qbar^{n,\varphi^n}(s), y \right) - G(Q(s), y)\right\|_{l^2}  |\psi^n(s,y)| \one_{\{|\psi^n |>(a(n)\sqrt{n})^{1/2} \}} dy\, ds\\
\leq & \int_{\X_T} 2 |\psi^n(s,y)| \one_{\{|\psi^n |>(a(n)\sqrt{n})^{1/2} \}} dy\, ds
\leq  \frac{2M\gamma_1\left( (a(n)\sqrt{n})^{-1/2} \right)}{a(n)\sqrt{n}}\\
\leq & \frac{8M}{a(n)\sqrt{n})^{1/2}},
\end{aligned}
\ee 
where the last inequality follows from the last sentence in Lemma \ref{lm:estimate} (b). Lemma \ref{lm:Gnqbar-Gq} tells that 
\be\label{eq:prop.Ec4}  
\begin{aligned}
\left\|  \Ec_4^{n,\varphi^n} \right\|_{*,T}\leq  &\frac{4\lambda d^2(a(n)\sqrt{n})^{1/2}}{n} \int_{[0,T]} \left\| Q(s) \right\|_{l^2} ds +\gamma_5(1+\frac{1}{n}) (a(n)\sqrt{n})^{1/2}\int_{[0,T]} \left\| \Qbar^{n,\varphi^n}(s)-Q(s) \right\|_{l^2} ds\\
\leq &\frac{4\lambda d^2(a(n)\sqrt{n})^{1/2}}{n} 2\|Q(0)\|_{l^2}^2 e^{8 (\lambda d+1)^2 T}T + \gamma_5 (1+\frac{1}{n}) \frac{(a(n)\sqrt{n})^{1/2}}{a(n)\sqrt{n}}\left\| Z^{n,\varphi^n} \right\|_{*,T}
\end{aligned}
\ee 
Since $\frac{(a(n)\sqrt{n})^{1/2}}{n} \to 0$ and $\frac{(a(n)\sqrt{n})^{1/2}}{a(n)\sqrt{n}} \to 0$, it follows from Lemma \ref{lm:Zbar.estimate} that $\E \left\|  \Ec_4^{n,\varphi^n} \right\|_{*,T} \to 0$ as $n \to \infty$.

In sum, 
\begin{align*}
\Zbar^{n,\varphi^n} (t) =&\Ec^{n,\varphi^n} (t) +\Btil^{n,\varphi^n}(t)+ \Ctil^{n,\varphi^n}(t)\\
=& \Ec^{n,\varphi^n} (t) +  \int_{[0,t]} Db(Q(s))\left[ \Zbar^{n,\varphi^n}(s) \right] ds + \int_{\X_t} G(Q(s), y) \psi^n(s,y) \one_{\{|\psi^n |\leq \beta a(n)\sqrt{n}\}} dy\, ds
\end{align*}
with $\Ec^{n,\varphi^n}= A^n+ M^{n,\varphi^n}+ \sum_{k=1}^4\Ec^{n,\varphi^n}_k \Rightarrow 0$. Thus, by continuous mapping theorem and Lemma \ref{lm:h}
\begin{align*}
\Gc^n(\frac{1}{n} N^{n\varphi^n}) = \Zbar^{n,\varphi^n} =& h\left( \Ec^{n,\varphi^n}, \psi^n\one_{\{|\psi^n |\leq \beta a(n)\sqrt{n}\}}\right)\\
\Rightarrow &h\left( 0, \psi\right)=\Gc(\psi),
\end{align*}
where the last equality follows from the definition of $\Gc$ as in \eqref{eq:Gc.def}, and the proof is complete.
\end{proof}

\section{Moderate deviation principle with buffer}\label{sec:buffer}
In this section we provide the MDP for the JSQ(d) system with buffer. We say the JSQ(d) system in Section \ref{sec:model} with buffer $K$ if the server is selected as in the no-buffer case, and a customer will immediately leave if he/she is sent to a server with queue length bigger than $K$. 

We take $\X=[0,2K (1\vee \lambda)]$ within the buffer $K$ setting. Set $l^2_K$ the subset of the space $l^2$ that contains $q=(q_0,q_1,...,)$ with $q_i=0$ for all $i\geq K+1$, and the spaces $\Smchat_K^n, \Smchat_K, \C([0,T]: l^2_K), \D([0,T]: l^2_K)$ are defined in a similar manner as $\Smchat^n, \Smchat, \C([0,T]: l^2), \D([0,T]: l^2)$. 
Moreover, $G^n_i(\cdot), b^n_i(\cdot), G_i(\cdot), b_i(\cdot)$ in \eqref{eq:Gn}, \eqref{eq:bn} and \eqref{eq:b.G} are taken to be 0 for $i\geq K+1$ within the buffer $K$ setting, and the $q_{i+1}$ terms are replaced by 0 in $b^n_K(\cdot)$, $b_K(\cdot)$ $G^n_K(\cdot)$ and $G_K(\cdot)$. Then we have that $G^n, G:\Smchat_K\times \X\to l^2_K$, $b^n, b:\Smchat_K\to l^2_K$.

By applying similar computations, all the estimates provided in Section \ref{sec:model} remain valid within the current buffer $K$ setting, provided that all $l^2$ norms are replaced with $l^2_K$ norms and all $i$ are restricted to $1\leq i\leq K$.

Analogous to Proposition \ref{prop:LLN}, the law of large numbers for the JSQ(d) system with buffer $K$ is given as follows.


\begin{Proposition}[LLN with buffer $K$]\label{prop:LLN.buffer} The following conclusions hold.
	\bi
	\item[\emph{(a)}] For each $n$ and $Q^n(0)\in l^2_K$, there exists a measurable mapping $\bar \Gc^n: \M\to \D([0, T]: l^2_K)$ such that $Q^n=(Q^n(t))_{t\in [0,T]}=\bar\Gc^n(\frac{1}{n} N^n)$.
	
	\item[\emph{(b)}] If $\|Q^n(0)-Q(0)\|_{l^2_K}\to 0$ with $Q^n(0), Q(0)\in l^2_K$, then $Q^n$ converges in probability to $Q\in \C([0,T]: l^2_K)$ in $\D([0,T]: l^2_K)$.
	\ei 
\end{Proposition}

The following provides the MDP for the JSQ(d) system with buffer $K$.

\begin{Theorem}[MDP with buffer $K$]\label{thm:MDP.buffer}
	Suppose $\|Q^n(0)-Q(0)\|_{l^2_K} \to 0$ with $Q^n(0) \in \Smchat^n_K, Q(0)\in \Smchat_K$. Then $\{a(n)\sqrt{n}(Q^n-Q)\}_n$ satisfies a LDP in $\D([0,T]: l^2_K)$ with speed $a^2(n)$ and rate function as 
	\be  
	I(\eta):= \inf_{\eta=\Gc_K(\psi)}\left\{   \frac{1}{2}\|  \psi \|_{L^2(\X_T)}\right\}
	\ee 
	where, for a given $\psi\in L^2(\X_T)$, $\Gc_K(\psi)$ defines the unique solution in $\C([0,T]: l^2_K)$ for the following ODE 
	\be\label{eq:Gc'}  
	\Gc_K(\psi)(t)=\int_0^t Db(Q(s)) \left[\Gc_K(\psi)(s) \right] ds+\int_{\X_t} G(Q(s),y)\psi(s,y)dy\,ds.
	\ee 
\end{Theorem}

The proof for Theorem \ref{thm:MDP.buffer} follows the same arguments as in Section \ref{sec:proof}. Especially, the indices $i$ are restricted to $1\leq i\leq K$, $l^2$ is replaced by $l^2_K$, and the spaces $\D([0, T]: l^2)$, $\C([0, T]: l^2)$ are adjusted to $\D([0, T]: l^2_K)$, $\C([0, T]: l^2_K)$, through the discussions in Section \ref{sec:proof.a}. Additionally, $\Smchat^n, \Smchat, l^2$ are replaced by $\Smchat_K^n, \Smchat_K, l^2_K$ through the discussions in Section \ref{sec:proof.b}.

\section{Convergence of rate functions: a case study}\label{sec:case}

In this section, we will study the convergence of rate function value as the buffer $K\to\infty$ in a special case. Given $q\in \C([0,T]: l^2)$  such that $q(0)=0$ and $I^{Q^*}(q)<\infty$, let $q^K$ be the truncation of $q$ with $q_i(t)\equiv 0$ for all $i\geq K+1$. Let $\lambda <1$. We denote by $Q^*$ (resp. $Q^*(K)$) the unique stationary solution to \eqref{eq:ODE} in the no buffer setting (resp. the unique solution in the buffer $K$ setting). To distinguish the notations, for any suitable $q\in \C([0,T]: l^2)$, we shall write $I^{Q^*}(\cdot)$ (resp. $I^{Q^*(K)}(\cdot)$) the value under $Q^*$ in the no buffer setting (resp. under $Q^*(K)$ in the buffer $K$ setting). 
We will investigate whether or not $I^{Q^*(K)}(q^K)\to I^{Q^*}(q)$ as $K\to \infty$. 

By using $Q^*_0=1$ and taking $b(q)=0$, i.e., 
\be\label{eq:station.0}
\lambda[(Q^*_{j-1})^d-(Q^*_{j})^d]= Q^*_j-Q^*_{j+1}\quad \forall j \geq 1,
\ee
we first get the explicit formula of $Q^*$ as 
\be \label{eq:stationaryQstar} 
Q^*_j=\lambda^{\frac{d^j-1}{d-1}}=\lambda^{\sum_{k=0}^{j-1} d^k},\quad i\geq 1.
\ee 
The first observation is that $Q^*$ and $Q^*(K)$ are close enough.
\begin{Lemma}\label{lm:Q.buffer}
For each $K\in \N$, we have that $0<Q^*_j(K)< Q^*_j$ for $1\leq j\leq K$. Moreover,
\be\label{eq:lm.QK}  
\frac{\max_{1\leq j\leq K} [Q^*_j-Q^*_j(K)]}{Q^*_K}\leq C (Q^*_K)^{d-1},
\ee 
where $C$ is a finite constant independent of $K$.
\end{Lemma} 
\begin{proof}
Take an arbitrary $K\in \N$. We first show by induction that
 \be\label{eq:buffer.1} 
Q^*_j(K) = \lambda(Q^*_{j-1}(K))^d- \lambda (Q^*_K(K))^d\quad 1\leq j\leq K.
 \ee
Notice that $\{ Q_j^*(K)\}_{1\leq j\leq K}$ satisfies
\begin{align}
&\lambda[(Q^*_{j-1}(K))^d-(Q^*_{j}(K))^d]= Q^*_j(K)-Q^*_{j+1}(K)\quad  \text{for }1\leq j\leq K-1, \label{eq:buffer}\\
& \lambda[(Q^*_{K-1}(K))^d-(Q^*_{K}(K))^d]= Q^*_K(K). \label{eq:bufferK}
\end{align}
Taking sum of above equalities gives 
\begin{align}
Q^*_1(K)=\lambda (Q^*_0(K))^d-\lambda (Q^*_K(K))^d=\lambda -\lambda (Q^*_K(K))^d, \label{eq:QK1}
\end{align} 
 where the last equality follows from $Q^*_0(K)=1$, so \eqref{eq:buffer.1} holds for $j=1$. Suppose \eqref{eq:buffer.1} holds for some $j\leq K-2$, then by \eqref{eq:buffer}, we have that 
 $$
 Q^*_{j+1}(K) =  Q^*_j(K)-\lambda(Q^*_{j-1}(K))^d+\lambda (Q^*_{j}(K))^d=\lambda (Q^*_{j}(K))^d- \lambda (Q^*_K(K))^d.
 $$
 So, \eqref{eq:buffer.1} holds for all $1\leq j\leq K$, where the case of $j=K$ directly follows from \eqref{eq:bufferK}. 
 Notice again that $Q^*_0(K)=Q^*_0=1$ and $Q^*_j=\lambda(Q^*_{j-1})^d$. Then applying induction again tells that $0<Q^*_j(K) <Q^*_j$ for all $1\leq j\leq K$. 
 
 Define $e(K):=\lambda (Q^*_K(K))^d$, 
and $e_j(K):= Q^*_j - Q^*_j(K) $. Then $e_j(K)>0$ for $1\leq j\leq K$. 
By \eqref{eq:buffer.1},
\be 
\begin{aligned}
Q_{j+1}^*(K)=&\lambda (Q^*_j)^d+\lambda (Q^*_j-e_j(K))^d-\lambda (Q^*_j)^d-e(K)\\
=& Q^*_{j+1}+\lambda \left[   (Q^*_j-e_j(K))^d-(Q^*_j)^d\right] -e(K)\\
\geq &Q^*_{j+1}-\lambda d(Q^*_j)^{d-1}e_j(K)-e(K)= Q^*_{j+1}- d \lambda^{d^j}e_j(K)-e(K),
\end{aligned}
\ee 
where the inequality follows from applying mean value theorem  on function $f(y)=y^2$ and the last equality follows from \eqref{eq:stationaryQstar}. Thus,
\bee  
0\leq e_{j+1}(K)\leq d \lambda^{d^j}e_j(K) + e(K)\quad 0\leq j\leq K-1 \text{ with } e_0=0,
\eee 
which, by induction, leads to 
\bee  
0\leq e_j(K)\leq \left( d^0+ d^1 \lambda^{d^{j-1}}+ d^2 \lambda^{d^{j-1}+d^{j-2}}+\dotsb + d^{j-1} \lambda^{d^{j-1}+\dotsb+ d^1}\right) e(K)= \left( \sum_{l=0}^{j-1} d^l \lambda^{\sum_{m=1}^l d^{j-m} } \right)e^K.
\eee
Notice that $ \sum_{l=0}^{j-1} d^l \lambda^{\sum_{m=1}^l d^{j-m} }\leq 1+ (j-1)d^{j-1} \lambda^{d^{j-1}}\leq 1+ (d^{j-1})^2 \lambda^{d^{j-1}}\leq M$, where $M:=\sup_{x\geq 0} x^2 \lambda^x+1<\infty$. Hence, by $e(K)=\lambda (Q^*_K(K))^d< \lambda(Q^*_K)^d$,
$$
\frac{\max_{1\leq j\leq K} |e_j(K)|}{Q^*_K}\leq M \frac{e(K)}{Q^*_K}\leq \lambda M (Q^*_K)^{d-1}.
$$
\end{proof}

To proceed the convergence analysis for the rate function values, we now provide more explicit formulas for $I^{Q^*}(q),I^{Q^*}(q^K)$ and $I^{Q^*(K)}(q^K)$ as follows. By \eqref{eq:station.0}, \eqref{eq:stationaryQstar} and Theorem \ref{thm:I},
\be  
\begin{aligned}
2I^{Q^*}(q)=& \|\psitil\|^2_{L^2}=\int_{[0,T]}\sum_{j\geq 1}\frac{\phi^2_j(t)}{4(Q_j^*-Q_{j+1}^*)^2}2(Q_j^*-Q_{j+1}^*) dt\\
=&\sum_{j\geq 1} \int_{[0,T]}\frac{\left( q'_j(t)-\left[ d\left(\lambda^{d^{j-1}}q_{j-1}(t)-\lambda^{d^i}q_j(t) \right)- (q_j(t)-q_{j+1} (t)) \right] \right)^2}{2(Q_j^*-Q_{j+1}^*)} dt.
\end{aligned}
\ee 
And thus,
\be\label{eq:QqK}  
\begin{aligned}
	2I^{Q^*}(q^K)
	= & \sum_{1\leq j\leq K-1} \int_{[0,T]}\frac{\left( q'_j(t)-\left[ d\left(\lambda^{d^{j-1}}q_{j-1}(t)-\lambda^{d^j}q_j(t) \right)- (q_j (t)-q_{j+1} (t)) \right] \right)^2}{2(Q_j^*-Q_{j+1}^*)} dt\\
	&\quad + \int_{[0,T]}\frac{\left( q'_K(t)-\left[ d\left(\lambda^{d^{K-1}}q_{K-1}(t)-\lambda^{d^K}q_K(t) \right)- (q_K (t)) \right] \right)^2}{2(Q_{K}^*-Q_{K+1}^*)} dt\\
		&\quad + \int_{[0,T]}\frac{d^2\left(\lambda^{d^{K}}q_{K}(t) \right)^2}{2(Q_{K+1}^*-Q_{K+2}^*)} dt,
\end{aligned}
\ee
By \eqref{eq:buffer}, \eqref{eq:bufferK} and Theorem \ref{thm:I}, we reach that
\begin{align}
&I^{Q^*(K)}(q^K) \notag\\
	&= \sum_{1\leq j\leq K-1} \int_{[0,T]}\frac{\left( q'_j(t)-\left[ d\lambda \left(  (Q^*_{j-1}(K))^{d-1}q_{j-1}(t)-(Q^*_{j}(K))^{d-1}q_j(t) \right)- (q_j (t)-q_{j+1} (t)) \right] \right)^2}{2(Q_j^*(K)-Q_{j+1}^*(K))} dt \notag\\
	&\quad + \int_{[0,T]}\frac{\left( q'_K(t)-\left[ d\lambda \left((Q^*_{K-1}(K))^{d-1}q_{K-1}(t)-(Q^*_{K}(K))^{d-1}q_K(t) \right)- (q_K (t)) \right] \right)^2}{2Q_{K}^*(K)} dt. \label{eq:QKqK}  
\end{align}

To investigate whether or not $I^{Q^*(K)}(q^K)\to I^{Q^*}(q)$, we provide the following equivalent and relatively simpler condition.

\begin{Lemma}\label{lm:qK}
	For any $q\in \C([0,T]: l^2)$ with $I^{Q^*}(q)<\infty$,
	\be\label{eq:eg.0}  
	\begin{aligned}
	&\lim_{K\to\infty}I^{Q^*(K)}(q^K) = I^{Q^*}(q) \Longleftrightarrow \lim_{K\to\infty} I^{Q^*}(q^K) = I^{Q^*}(q) \\
	& \Longleftrightarrow \lim_{K\to\infty} \int_{[0,T]}\frac{\left| q'_K(t)+(q_K (t)-q_{K+1}(t)) \right|^2}{2Q_{K}^*} dt = 0.
	\end{aligned}
	\ee 
\end{Lemma}

\begin{proof}
We first prove the second equivalence statement in \eqref{eq:eg.0}. Since $0\leq q_j\leq 1$, by \eqref{eq:stationaryQstar}, the last term on the right-hand side of \eqref{eq:QqK} tends to 0 as $K\to\infty.$ 
Notice that 
\be\label{eq:lm.qQterm} 
\lambda^{d^{K-1}}/\sqrt{Q^*_K} = \lambda^{\frac{1}{2}[d^{K-1}-\frac{d^{K-1}-1}{d-1}]}\leq \lambda^{1/2}\quad \forall K\in \N.
\ee 
Then $\frac{|\lambda^{d^{K-1}}q_{K-1}(t)-\lambda^{d^K}q_K(t)|}{\sqrt{Q^*_K-Q^*_{K+1}}}\leq |q_{K-1}(t)|$ for $K$ large enough.
Then, by $I^{Q^*}(q)<\infty$, the second line of \eqref{eq:QqK} tends to 0 as $K\to\infty$ if and only if $\lim_{K\to\infty} \int_{[0,T]}\frac{\left| q'_K(t)+(q_K (t)-q_{K+1}(t)) \right|^2}{2Q_{K}^*} dt = 0$. Thus, $I^{Q^*}(q^K) \to I^{Q^*}(q)$ if and only if $\lim_{K\to\infty} \int_{[0,T]}\frac{\left| q'_K(t)+(q_K (t)-q_{K+1}(t)) \right|^2}{2Q_{K}^*} dt = 0$.

Now we prove the first equivalence statement in \eqref{eq:eg.0}. 
By \eqref{eq:stationaryQstar} and Lemma \ref{lm:Q.buffer}, 
for $K$ big enough,
\be\label{eq:lmqK.0} 
1-\eps_K
\leq \left\{\sqrt{\dfrac{Q^*_j-Q^*_{j+1}}{Q^*_j(K)-Q^*_{j+1}(K)}} \right\}_{1\leq j\leq K-1}
\text{ and } \dfrac{Q^*_K-Q^*_{K+1}}{Q^*_K(K)}
\leq 1+\eps_K\\
\ee 
where $0 \leq \eps_K\leq C(Q^*_K)^{d-1}$ for some finite constant $C$ independent of $K$.
For $K$ large enough, we have 
\begin{align*}
&\dfrac{\left| \lambda\left( (Q^*_{j-1}(K))^{d-1}q_{j-1}(t)-(Q^*_{j}(K))^{d-1}q_j(t) \right)-\left(\lambda^{d^{j-1}}q_{j-1}(t)-\lambda^{d^j}q_j(t) \right) \right| }{\sqrt{Q^*_j-Q^*_{j+1}}}\\
& \leq  4 \frac{|\lambda  (Q^*_{j-1}(K))^{d-1}- \lambda(Q^*_{j-1})^{d-1}|}{\sqrt{Q^*_j-Q^*_{j+1}}}\leq 4 (d-1) \lambda \frac{|Q^*_{j-1}(K)- Q^*_{j-1}|}{\sqrt{Q^*_j-Q^*_{j+1}}} \\
&\leq 8 (d-1) \lambda \frac{|Q^*_{j-1}(K)-  Q^*_{j-1}|}{\sqrt{Q^*_K}}\leq  8 (d-1) C  \lambda (Q^*_K)^{d-1/2}\quad 1\leq j\leq K,
\end{align*}
where the last inequality follows from Lemma \ref{lm:Q.buffer}. Then by denoting the numerators in the squares of the first $K-1$ integrals on the right-hand side of \eqref{eq:QqK} (resp. \eqref{eq:QKqK}) as $\{ A_i\}_{1\leq i\leq K-1}$ (resp. as $\{ \tilde A_i\}_{1\leq i\leq K-1}$), we have for $K$ large enough that 
\begin{align}
\frac{|A_j-\tilde A_j|}{\sqrt{Q^*_j-Q^*_{j+1}}}\leq 8 d C \lambda (Q^*_K)^{d-1/2}=: \tilde \eps_K\quad 
1\leq j\leq K-1. \label{eq:lmqK.1}
\end{align}
By writing
\begin{align*}
\frac{\tilde{A}_j (t) }{\sqrt{ Q^*_j(K)-Q^*_{j+1}(K)} } = \left( \frac{ A_j(t) }{ \sqrt{Q^*_j-Q^*_{j+1}} } + \frac{ (\tilde A_j(t)-A_j(t)) }{ \sqrt{Q^*_j-Q^*_{j+1}} } \right)\frac{ \sqrt{Q^*_j-Q^*_{j+1}} }{ \sqrt{ Q^*_j(K)-Q^*_{j+1}(K) } }
\end{align*}
and combining with \eqref{eq:lmqK.0} and \eqref{eq:lmqK.1}, we have for $K$ large enough that 
\begin{align*}
&\left| \left(\frac{A_j(t)}{\sqrt{Q^*_j-Q^*_{j+1} }} \right)^2 -\left(  \frac{\tilde A_j(t)}{ \sqrt{Q^*_j(K)-Q^*_{j+1}(K) }} \right)^2  \right| \\
&\leq   \left[ \eps_K \frac{A_j(t)}{\sqrt{Q^*_j-Q^*_{j+1} }} +  \tilde\eps_K \left( 1+ \eps_K\right) \right]
\cdot   \left[ 2 \frac{A_j(t)}{\sqrt{Q^*_j-Q^*_{j+1} }} + \eps_K \frac{A_j(t)}{\sqrt{Q^*_j-Q^*_{j+1} }} +  \tilde\eps_K \left( 1+ \eps_K\right) \right] \\
& = \eps_K (2+\eps_K)\left(\frac{A_j(t)}{\sqrt{Q^*_j-Q^*_{j+1} }} \right)^2 +2\tilde \eps_K (1+ \eps_K)^2 \frac{A_j(t)}{\sqrt{Q^*_j-Q^*_{j+1} }}+\tilde \eps_K^2 (1+\eps_K)^2\\
&\leq [\eps_K(2+\eps_K)+\tilde \eps_K] \left(\frac{A_j(t)}{\sqrt{Q^*_j-Q^*_{j+1} }} \right)^2 +\tilde\eps_K\left[ \tilde \eps_K (1+\eps_K)^2+(1+\eps_K)^4 \right] \quad \text{for}\; 1\leq j\leq K.
\end{align*}
Notice that both $\eps_K$ and $K\tilde\eps_K$ tend to 0 as $K\to\infty$. Then taking the difference between \eqref{eq:QqK} and \eqref{eq:QKqK} and combining with the above inequality yields
\begin{align*}
&\left| I^{Q^*}(Q^K) - I^{Q^*(K)}(Q^K)\right|\\
& \leq \sum_{1\leq j\leq K-1} \int_{[0,T]} \left| \left(\frac{A_j(t)}{\sqrt{Q^*_j-Q^*_{j+1} }} \right)^2 -\left(  \frac{\tilde A_j(t)}{ \sqrt{Q^*_j(K)-Q^*_{j+1}(K) }} \right)^2  \right|  dt \\
&\quad +\left| \left(\frac{A_K(t)}{\sqrt{Q^*_K-Q^*_{K+1} }} \right)^2 -\left(  \frac{\tilde A_K(t)}{ \sqrt{Q^*_K(K)}} \right)^2  \right|+  \int_{[0,T]}\frac{d^2\left(\lambda^{d^{K}}q_{K}(t) \right)^2}{2(Q_{K+1}^*-Q_{K+2}^*)} dt\\
& \leq [\eps_K(2+\eps_K)+\tilde \eps_K] I^{Q^*}(q) + TK\tilde\eps_K\left[ \tilde \eps_K (1+\eps_K)^2+(1+\eps_K)^4 \right]\\
& \qquad +  \int_{[0,T]}\frac{d^2\left(\lambda^{d^{K}}q_{K}(t) \right)^2}{2(Q_{K+1}^*-Q_{K+2}^*)} dt\\
&\to 0 \text{ as } K \to \infty,
\end{align*}
which completes the proof.
\end{proof}

Given $q\in \C([0,T]: l^2)$, we say $q$ belongs to family $\Ac$ if there exists a finite disjoint partition $[0,T] = \cup_{i=1}^n E_i$ for some $n\in \N$ such that all $q_k$ are monotone on each $E_i$, we say $q$ belongs to family $\Bc$ if there exists a countable disjoint partition $[0,T] = \cup_{i=1}^\infty E_i$ such that on each subinterval $E_i = [a_i,b_i]$, 
the signs of $q_k(s)-q_k(a_i)$ and $\int_{a_i}^s q_k(t)dt$ are the same and do not change over $s \in E_i$.

The following provides the convergence of rate functions restricted to families $\Ac$ and $\Bc$.

\begin{Proposition}
\label{prop:convergence-I}
Given $q\in \Ac\cup \Bc$ with $I^{Q^*}(q)<\infty$, then
$$
\lim_{K\to\infty} I^{Q^*(K)}(q^K) = I^{Q^*}(q).
$$
\end{Proposition}

\begin{proof}
Take $q\in \C([0,T]: l^2)$ with $I^{Q^*}(q)<\infty$. By Lemma \ref{lm:qK}, it suffices to prove
\bee
\lim_{K\to\infty} \int_{[0,T]}\frac{\left| q'_K(t)+(q_K (t)-q_{K+1}(t)) \right|^2}{2Q_{K}^*} dt = 0.
\eee 
for $\Ac$ and $\Bc$ separately.

\underline{For $q\in \Ac$:} Define 
\be\label{eq:eg.prop1} 
T_0:=\sup\left\{ t\leq T: \sum_{j=1}^\infty \frac{\max_{0\leq s\leq t}q^2_j(s)}{Q^*_j}<\infty \right\}.
\ee 
Since $I(q)<\infty$, above also implies that  $$\sum_{j=1}^\infty\frac{\max_{0\leq s\leq T_0}(q'_j(s))^2}{Q^*_j}<\infty.$$
If $T_0=T$, then the desired result is achieved. Suppose $T_0<T$. Then consider 
\be\label{eq:eg.prop3}  
\delta:=\sup\{0\leq s\leq T-T_0: \text{ all $q_k$ are monotone on $[T_0, T_0+s]$}\}.
\ee 
Since $q\in \Ac$, we must have $0<\delta\leq T-T_0$. Set $\bar q_{j}(t) = q_j(t)-q_j(T_0)$ on $[T_0, T_0+\delta]$ for all $j\in \N$. By \eqref{eq:eg.prop1} and \eqref{eq:eg.prop3}, there exists a sequence $\{c_j\}_{j\in \N}$ with $\sum_{j\in \N} c_j^2<\infty$ such that 
\bee
\begin{aligned}
|q_j'(t)+q_j(t)-q_{j+1}(t)|- 2  c_{j} \sqrt{Q^*_j}\leq &   |\bar q_K'(t)+\bar q_j(t)-\bar q_{j+1}(t)|\\
	\leq & |q_j'(t)+ q_j(t)- q_{j+1}(t)| +2 c_{j} \sqrt{Q^*_j} \quad \forall t\in[T_0, T_0+\delta].
 \end{aligned}
\eee 
Then
\be\label{eq:eg.prop5}
	\sum_{j=0}^\infty \int_{T_0}^{T_0+\delta} \frac{|\bar q_j'(t)+\bar q_j(t)-\bar q_{j+1}(t)|^2}{2Q^*_j} dt
	\leq 2	\sum_{j=0}^\infty\int_{[0,T]} \left(\frac{|q_j' (t)+q_j(t)-q_{j+1}(t)|^2}{2Q^*_j}+2  c^2_{j}\right) dt<\infty.
\ee
Thus, we find another sequence $\{\tilde c_j\}_{j\in \N}$ with $\sum_{j\in \N} \tilde c_j^2<\infty$ such that $|\bar q_j'(t)+\bar q_j(t)-\bar q_{j+1}(t)|=\tilde c_j\sqrt{2Q^*_j}$.
Since $Q_K^*$ is doubly exponentially decaying, we choose $K_0$ large enough so that
$$\sum_{k=K}^\infty \tilde c_k \sqrt{2Q_k^*} \le 
2(\sup_{k \ge K}\tilde c_k)
\sqrt{2Q_K^*} \leq (\tilde c_K+\lambda^{d^{K-1}/2})\sqrt{6Q^*_K}
$$
for all $K \ge K_0$.

Now we analyse $\bar q(t)$ on $[T_0, T_0+\delta]$. Notice that $\bar q_j(t)$ start with $0$ at $t=T_0$ and $\bar q_j(t), \bar q_j'(t)$ have the same sign on $[T_0, T_0+\delta]$. Take an arbitrary $K\geq K_0$. If $\bar q_j(t)$ has the same sign as $\bar q_K(t)$ on $[T_0, T_0+\delta]$ for all $j > K$, then by \eqref{eq:eg.prop5}, we can take sequence $\{a_j\}_{j\geq K}$ with $|a_j|=\tilde c_j$ such that 
\begin{equation}
	\label{eq:q_K_cvg_1}
	\bar q_j(T_0+\delta)+\int_{[T_0,T_0+\delta]}\bar q_j(t)dt -\int_{[T_0,T_0+\delta]}\bar q_{j+1}(t)dt = a_j \sqrt{2 Q_{j}^*}\quad \forall j\geq K.
\end{equation}
Taking sum gives that 
\begin{align}
&\sum_{j=K}^\infty 	\left| \bar q_j(T_0+\delta) \right| +\int_{[T_0,T_0+\delta]} |\bar q_K(t)|dt \leq  \sum_{j=K}^\infty |a_j| \sqrt{2Q^*_j}\leq 2(\sup_{j \ge K}|a_j|)\sqrt{2Q^*_K} \notag\\
&\leq  (\tilde c_K+\lambda^{d^{K-1}/2})\sqrt{6Q^*_K} \quad \text{for $K\geq K_0$ large enough}, \label{eq:eg.1} 
\end{align}
If $\bar q_j(t)$ has the opposite sign as $q_K(t)$ for some $j > K$, then let $$n := \min\{j> K: q_j(t) \text{ has the opposite sign as } q_K(t)\}.$$
Taking sums in \eqref{eq:q_K_cvg_1} for $K\leq j\leq n-1$ gives
$$\sum_{j=K}^{n-1} \bar q_j(T_0+\delta) + \int_{[T_0,T_0+\delta]}\bar q_K(t)dt - \int_{[T_0,T_0+\delta]}\bar q_{n}(t)dt = \sum_{j=K}^{n-1} a_j \sqrt{2 Q_{j}^*}.$$
By the definition of $n$, the left three terms (including the negative sign) have the same sign.
Therefore
\begin{align*}
&\int_{[T_0,T_0+\delta]} |q'_K(t)|dt=|q_K(T)| \le \sum_{j=K}^{n-1} |a_j| \sqrt{2 Q_{j}^*} \le 
2(\sup_{j \ge K}c_j) \sqrt{2 Q_{j}^*} \\
&\leq  (\tilde c_K+\lambda^{d^{K-1}/2})\sqrt{6Q^*_K} \quad \text{for $K\geq K_0$ large enough}
\end{align*}

Notice that $\sum_{j\in \N} (\tilde c_K+\lambda^{d^{K-1}/2})^2<\infty$. Hence, no matter which the case it is, we have either $\sum_{j\in \N} \int_{T_0}^{T_0+\delta}\frac{|\bar q_j(t)|^2}{Q^*_j} dt<\infty$ or $\sum_{j\in \N} \int_{T_0}^{T_0+\delta}\frac{|\bar q'_j(t)|^2}{Q^*_j} dt<\infty$, which implies that 
$$\text{Both }\sum_{j=0}^\infty \int_{T_0}^{T_0+\delta}\frac{|\bar q_j(t)|^2}{Q^*_j} dt<\infty \text{ and } \sum_{j=0}^\infty \int_{T_0}^{T_0+\delta}\frac{|\bar q'_j(t)|^2}{Q^*_j} dt<\infty.$$
Thus, $T_0+\delta$ also satisfies the definition of $T_0$ in \eqref{eq:eg.prop1}, which is a contradiction. Hence, $T_0=T$.

\underline{For $q\in \Bc$:} 
	Suppose there exists a countable partition $[0,T] = \cup_{i=1}^\infty E_i$ such that on each subinterval $E_i = [a_i,b_i]$, 
	the signs of $q_k(s)-q_k(a_i)$ and $\int_{a_i}^s q_k(t)dt$ are the same and do not change over $s \in E_i$.
	Note that for each $E_i$ and $s \in E_i$, for any fixed $p \in (0,\frac{1}{2})$,
	\begin{align*}
		\left| q_K(s)-q_K(a_i) + \int_{a_i}^{s}q_K(t)dt - \int_{a_i}^{s}q_{K+1}(t)dt \right|
		& \le (2 Q_{K}^*)^p \int_{a_i}^{s} \frac{\left| q'_K(t)+q_K(t) -q_{K+1}(t)  \right|}{(2 Q_{K}^*)^p}dt.
	\end{align*}
	Therefore
	\begin{align*}
		& \left| \sum_{k=K}^{n-1} (q_k(s)-q_k(a_i)) + \int_{a_i}^{s}q_K(t)dt - \int_{a_i}^{s}q_n(t)dt \right| \\
		& \le \sum_{k=K}^{n-1} (2 Q_{k}^*)^p \int_{a_i}^{s} \frac{\left| q'_k(t)+q_k(t) -q_{k+1}(t)  \right|}{(2 Q_{k}^*)^p}dt \\
		& \le (2 Q_{K}^*)^p \sum_{k=K}^\infty \int_{a_i}^{s} \frac{\left| q'_k(t)+q_k(t) -q_{k+1}(t)  \right|}{(2 Q_{k}^*)^p}dt.
	\end{align*}
	So the above arguments implies that, for any $K \ge 1$,
	\bee\label{eq:case1} 
	|q_K(s)-q_K(a_i)| + \left|\int_{a_i}^{s}q_K(t)dt\right| \le (2 Q_{K}^*)^p \sum_{k=K}^\infty \int_{a_i}^{s} \frac{\left| q'_k(t)+q_k(t) -q_{k+1}(t)  \right|}{(2 Q_{k}^*)^p}dt.
	\eee 
	This means, for every $s \in [0,T]$
	$$|q_K(s)| + \left|\int_0^s q_K(t)dt\right| \le (2 Q_{K}^*)^p \sum_{k=K}^\infty \int_0^T \frac{\left| q'_k(t)+q_k(t) -q_{k+1}(t)  \right|}{(2 Q_{k}^*)^p}dt.$$
	Note that the right-hand side of the above inequality vanishes as $K \to \infty$:
	\begin{align}
		& \sum_{k=K}^\infty \int_0^T \frac{\left| q'_k(t)+q_k(t) -q_{k+1}(t)  \right|}{(2 Q_{k}^*)^p}dt \notag \\
		& = \sum_{k=K}^\infty \int_0^T \frac{\left| q'_k(t)+q_k(t) -q_{k+1}(t)  \right|}{\sqrt{2 Q_{k}^*}} (2 Q_{k}^*)^{1/2-p}dt \notag \\
		& \le \left(\sum_{k=K}^\infty \int_0^T \frac{\left| q'_k(t)+q_k(t) -q_{k+1}(t)  \right|^2}{2 Q_{k}^*}dt\right)^{1/2} \left(\sum_{k=K}^\infty \int_0^T (2 Q_{k}^*)^{1-2p}dt\right)^{1/2} \notag \label{eq:case2},
	\end{align}
    where the first term in the last line above is the square root of the tail of $I^{Q^*}(q)$ and the second term is finite. 
    The result follows.
\end{proof}

\begin{Remark}
We note that $\Ac\setminus \Bc, \Bc\setminus \Ac, (\Ac\cup \Bc)^c$ are all non-empty.
Take a sequence $\{c_j\}_{j\in \N}$ such that $\sum_{j\in \N}c^2_j<\infty$ and let $T=2$.
\bi   
\item The following $q$ with $I^{Q^*}(q)<\infty$ satisfies $q\in \Ac\setminus \Bc$:
$$
q_j(t) = c_j\sqrt{Q^*_j} (1-\frac{1}{j})t, \; t\in[0,1]; \quad  q_j(t) = c_j\sqrt{Q^*_j} (1-\frac{1}{j}-t), \; t\in[1,2].
$$
\item The following $q$ with $I^{Q^*}(q)<\infty$ satisfies $q\in \Bc\setminus \Ac$:
$$
q_j(t) = c_j\sqrt{Q^*_j} t, \; t\in[0,1/j]; \;\;  q_j(t) = c_k\sqrt{Q^*_j} (1/j-t),\; t\in[1/j, 2/j]; \;\; q_j(t)=0, \; t\in[2/k, 2].
$$
\item The following $q$ with $I^{Q^*}(q)<\infty$ satisfies $q\notin \Ac\cup \Bc$:
$$
q_j(t) = c_j\sqrt{Q^*_j} t, \; t\in[0,1/j]; \quad  q_j(t) = c_j\sqrt{Q^*_j} (1/j-t), \; t\in[1/j, 2].
$$
Indeed, we can see that for each $q_j$, there can not be a partition point on the interval $[1/j, 2/j]$, then $(0,1/4)\subset \cup_j [1/j, 2/j]$ tells that no partition points on $(0,1/4)$. Hence, $q\notin \Ac\cup \Bc$.
\ei 
\end{Remark}

\begin{Remark}
    The result in Proposition \ref{prop:convergence-I} is only a case study, which is already non-trivial due to the non-linear dependence of the rate function $I^{Q^*}(q)$ on the infinite dimensional LLN trajectory $Q^*$ and the given trajectory $q$.
    It remains an open and interesting question whether 
    $$
    \lim_{K\to\infty} I^{Q^*(K)}(q^K) = I^{Q^*}(q)
    $$
    for all $q$ with $I^{Q^*}(q)<\infty$, and whether it holds for initial conditions different from the equilibrium state $Q^*$.
\end{Remark}
 
\appendix

\section{Alternative proof of Corollary \ref{cor:mu-MDP}}
\label{sec:exp-equiv}

In this section we give an alternative proof of Corollary \ref{cor:mu-MDP} without using Theorem \ref{thm:MDP}, as mentioned in Remark \ref{rmk:exp-equiv}, via an exponential equivalence argument combined with the MDP result in \cite{BudhirajaWu2017moderate}. 

Here we extend $l^2$ by adding the $0$-th component, and let $\mu^n_i(t)$ be the proportion of length $i$ at time $t$ for all $i\geq 0$. Let $\ebd_i$ denote the $i$-th unit vector for all $i=0,1,2,\dotsc$ We will make use of Girsanov's Theorem and estimate the quadratic variation associated with the change-of-measure density to show exponential equivalence (see e.g.\ \cite[Lemma 4.1]{CoppiniDietertGiacomin2019law} for its application in the case of mean-field type interacting diffusions). 

Define for $i\geq 0$ that
\be\label{eq:tildeRn}  
\begin{aligned}
\tilde R^n_i(\tilde\eta):=  \left[
\left(  \begin{matrix}
	n\left(\sum_{j\geq i}\tilde\eta_j \right)\\
	d
\end{matrix}\right)
-
\left( \begin{matrix}
	n\left(\sum_{j\geq i+1}\tilde\eta_j \right)\\
	d
\end{matrix}\right)
\right]\bigg/ 
\left(  \begin{matrix}
	n\\
	d
\end{matrix}\right), 
\tilde R_i(\tilde\eta):=  \left(\sum_{j\geq i}\tilde\eta_j \right)^d-\left(\sum_{j\geq i+1}\tilde\eta_j \right)^d,  
\end{aligned}
\ee
and let
\be\label{eq:tildeGn}  
\begin{aligned}
\tilde G^n(\tilde\eta,y) :=&\sum_{i \ge 0}  \tilde G^n_i(\tilde\eta, y) \ebd_i :=
 \left\{-\one_{[0, \lambda \tilde R^n_{0}(\tilde\eta))}(y)   +\one_{[\lambda_0, \lambda_0 + \tilde\eta_{1}  )}(y) \right\} \ebd_0\\
& \qquad\qquad \qquad \qquad + \sum_{i\geq 1} \bigg\{ \one_{[2(i-1)\lambda_0, 2(i-1)\lambda_0+\lambda \tilde R^n_{i-1}(\tilde\eta))}(y)-\one_{[2i\lambda_0, 2i\lambda_0+\lambda \tilde R^n_{i}(\tilde\eta))}(y)   \\
& \qquad\qquad \qquad \qquad \qquad\quad - \one_{[(2i-1)\lambda_0, (2i-1)\lambda_0 + \tilde\eta_i  )}(y) +\one_{[(2i+1)\lambda_0, (2i+1)\lambda_0 + \tilde\eta_{i+1}  )}(y) \bigg\} \ebd_i,\\
\tilde G(\tilde\eta,y) :=& \sum_{i \ge 0}  \tilde G_i(\tilde\eta, y) \ebd_i := \left\{-\one_{[0, \lambda \tilde R_{0}(\tilde\eta))}(y)   +\one_{[\lambda_0, \lambda_0 + \tilde\eta_{1}  )}(y) \right\} \ebd_0\\
& \qquad\qquad \qquad \qquad + \sum_{i\geq 1} \bigg\{ \one_{[2(i-1)\lambda_0, 2(i-1)\lambda_0+\lambda \tilde R_{i-1}(\tilde\eta))}(y) - \one_{[2i\lambda_0, 2i\lambda_0+\lambda \tilde R_{i}(\tilde\eta))}(y)  \\
& \qquad \qquad \qquad \qquad \qquad\quad - \one_{[(2i-1)\lambda_0, (2i-1)\lambda_0 + \tilde\eta_i)}(y) + \one_{[(2i+1)\lambda_0, (2i+1)\lambda_0 + \tilde\eta_{i+1})}(y) \bigg\} \ebd_i.
\end{aligned}
\ee  
 Then $\mu^n$ follows the dynamic
 \be\label{eq:JSQd.mun} 
\begin{aligned}
	\mu^n(t) & = \mu^n(0) +\frac{1}{n} \int_{\X_t} \tilde G^n(\mu^n(s-),y) \, N^{n}(dy\, ds) \\
	& = \mu^n(0) + \int_{[0,t]}\tilde b^n(\mu^n(s))ds+ \frac{1}{n} \int_{\X_t} \tilde G^n(\mu^n(s-),y) \, \tilde N^{n}(dy\, ds) 
\end{aligned}
\ee 
with
$
\tilde b^n(\tilde\eta):=\sum_{i\geq 0}\tilde b^n_i(\tilde \eta) \ebd_i:= \left(-\lambda\tilde R^n_{0}(\tilde\eta)+ \tilde\eta_{1} \right)\ebd_0+\sum_{i \ge 1} \left( \lambda \tilde R^n_{i-1}(\tilde\eta)-\lambda\tilde R^n_{i}(\tilde\eta)-\tilde\eta_i + \tilde\eta_{i+1}\right) \ebd_i
$.

We consider the following auxiliary model 
 \be\label{eq:JSQd.mun'} 
\begin{aligned}
	\bar\mu^n(t) & = \bar\mu^n(0) +\frac{1}{n} \int_{\X_t} \tilde G(\bar \mu^n(s-),y) \, N^{n}(dy\, ds) \\
	& = \bar\mu^n(0) + \int_{[0,t]}\tilde b(\bar\mu^n(s))ds+ \frac{1}{n} \int_{\X_t} \tilde G(\bar \mu^n(s-),y) \, \tilde N^{n}(dy\, ds) 
\end{aligned}
\ee 
with
$
\tilde b(\tilde\eta):= \sum_{i\geq 0}\tilde b_i(\tilde \eta) \ebd_i := \left(-\lambda\tilde R_{0}(\tilde\eta)+ \tilde\eta_{1} \right)\ebd_0+ \sum_{i\geq 1} \left( \lambda \tilde R_{i-1}(\tilde\eta)-\lambda\tilde R_i(\tilde\eta)-\tilde\eta_i +\tilde\eta_{i+1}\right) \ebd_i
$ and $\bar \mu^n(0)= \mu^n(0)$.
Notice that $\tilde b(\tilde \eta, y) =\left(1, b(\eta, y)\right)$ for all $y$ whenever $\eta_i=\sum_{j\geq i}\tilde \eta_i, \forall i.$
Then by \cite[Theorem 3.2]{BudhirajaWu2017moderate}, $\{a(n)\sqrt{n}(\bar\mu^n-\mu)\}$ satisfies a LDP with speed $a^2(n)$ with rate function $\tilde I$ defined in \eqref{cor:mu-MDP}.

Now we adapt the argument in \cite[Lemma 4.1]{CoppiniDietertGiacomin2019law} to show exponential equivalence between $\mu^n$ and $\bar \mu^n$, which leads to the desired result. 
By \eqref{eq:tildeGn}---\eqref{eq:JSQd.mun'}, consider the Radon-Nikodym derivative (associated with $\mu^n$ and $\bar \mu^n$)
$$
\Ec_T  
:=\exp\left[\sum_{i=0}^\infty\int_{[0,\infty)\times[0,T]}\log \frac{\tilde R^n_i(\bar \mu^n(s-))}{\tilde R_i(\bar \mu^n(s-))} 
	\one_{[2i\lambda_0, 2i\lambda_0+\lambda \tilde R_i(\bar \mu^n(s-)))}(y) N(dsdy) \right].
$$
Note that there is no time integration in the exponent because the sum of rates appearing in \eqref{eq:tildeGn} are the same:
$$
\sum_{i\geq 0} \left( \lambda \tilde R^n_i(\tilde\eta) + \tilde\eta_i \right) = \lambda + \sum_{i\geq 0} \tilde\eta_i = \sum_{i\geq 0} \left( \lambda \tilde R_i(\tilde\eta) + \tilde\eta_i \right) \quad \forall \tilde \eta.
$$
Write $Y^n=a(n)\sqrt{n}(\mu^n-\mu)$ and $\Ybar^n=a(n)\sqrt{n}(\bar\mu^n-\mu)$.
Then given a closed set $A \subset \D([0,T]: l^2)$, it follows from Girsanov's Theorem and Holder's inequality (with $1/p+1/q=1$) that
\begin{equation*}
	\P(Y^n \in A) = \E [1_{\{\Ybar^n \in A\}} \Ec_T] \le [\P(\Ybar^n \in A)]^{\frac{1}{p}} [\E(\Ec_T^q)]^{\frac{1}{q}}.
\end{equation*}
For the rest part of this proof, we write $R^n_i(s-)$ and $R_i(s-)$ short for $\tilde R^n_i(\bar \mu^n(s-))$ and  $\tilde R_i(\bar \mu^n(s-))$ respectively, and omit writing $\one_{\{R_i(s)>0\}}$ when $R_i(s)$ appears in the denominator.
Also write $B_i(s-)=[2i\lambda_0, 2i\lambda_0+\lambda \tilde R_i(\bar \mu^n(s-)))$.
By Cauchy-Schwarz inequality,
\begin{align*}
    \E(\Ec_T^q) & = \E\left[\exp\left(\sum_{i=0}^\infty\int_{[0,\infty)\times[0,T]}\log \frac{( R^n_i(s-) )^q}{(R_i(s-))^q} \one_{B_i(s-)}(y) N(dsdy) \right) \right] \\
	& = \E\bigg\{ \exp\left[\frac{1}{2} \sum_{i=0}^\infty \int_0^T \left( \frac{(R^n_i(s))^{2q}}{(R_i(s))^{2q-1}} - R_i(s) \right)ds\right]\\
	&\qquad  \cdot 	\exp\bigg[ -\frac{1}{2} \sum_{i=0}^\infty \int_0^T \left( \frac{(R^n_i(s))^{2q}}{(R_i(s))^{2q-1}} - R_i(s) \right)ds \\
    & \qquad \qquad +\sum_{i=0}^\infty\int_{[0,\infty)\times[0,T]}\log \frac{( R^n_i(s-) )^q}{(R_i(s-))^q} \one_{B_i(s-)}(y) N(dsdy) \bigg]  \bigg\}\\
	&\le \left\{\E\bigg[ \exp\left( \sum_{i=0}^\infty \int_0^T \left( \frac{(R^n_i(s))^{2q}}{(R_i(s))^{2q-1}} - R_i(s) \right)ds\right) \right\}^{1/2}\\
	&\qquad \cdot \bigg\{  
	\E\bigg[\exp\bigg( -\sum_{i=0}^\infty \int_0^T \left( \frac{(R^n_i(s))^{2q}}{(R_i(s))^{2q-1}} - R_i(s) \right)ds\\
	&\qquad \qquad\qquad +\sum_{i=0}^\infty\int_{[0,\infty)\times[0,T]}\log \frac{(R^n_i(s-))^{2q}/(R_i(s-))^{2q-1}}{R_i(s-)} 
	\one_{B_i(s-)}(y) N(dsdy) \bigg)  \bigg]
	\bigg\}^{1/2}
\end{align*}
Here the last expectation in the last line corresponds to a Radon-Nikodym derivative and hence is $1$. Therefore,
$$
\P(Y^n \in A)\leq [\P(\Ybar^n \in A)]^{\frac{1}{p}} \left\{\E\left[ \exp\left( \sum_{i=0}^\infty \int_0^T \left( \frac{(R^n_i(s))^{2q}}{(R_i(s))^{2q-1}} - R_i(s) \right)ds\right) \right]\right\}^{1/2q}.
$$
In particular, we have that for any $p,q$ with $1/p+1/q=1$,
\begin{align*}
& \limsup_{n\to\infty} a^2(n) \log \P(Y^n \in A) \\
& \le - \frac{1}{p} \inf_{x\in A} \Itil(x) + \frac{1}{2q}	\limsup_{n\to\infty} a^2(n) \log \E\left[ \exp\left( \sum_{i=0}^\infty \int_0^T \left( \frac{(R^n_i(s))^{2q}}{(R_i(s))^{2q-1}} - R_i(s) \right)ds\right) \right] \le - \frac{1}{p} \inf_{x\in A} \Itil(x),
\end{align*}
where the second inequality follows from Lemma \ref{lm:Rnbound} below. Since this holds for arbitrary $p>1$, we conclude that 
$$
\limsup_{n\to\infty} a^2(n) \log \P(Y^n \in A)\le- \inf_{x\in A} \Itil(x).
$$
A similar argument shows that 
$$
\liminf_{n\to\infty} a^2(n) \log \P(Y^n \in A)\ge- \inf_{x\in A} \Itil(x),\quad \text{for any open set } A \subset \D([0,T]: l^2).
$$
This completes the proof. \qed

\begin{Lemma}\label{lm:Rnbound}
For any $q > 1$, we have that
$$
\lim_{n\to\infty} a^2(n) \log \E\left[ \exp\left( \sum_{i=0}^\infty \int_0^T \left( \frac{(R^n_i(s))^{2q}}{(R_i(s))^{2q-1}} - R_i(s) \right)ds\right) \right]=0
$$
\end{Lemma}

\begin{proof}
Omitting the dependence of $\tilde{R}_i^n(\etatil)$ and $\tilde{R}_i(\etatil)$ on $\etatil$ in \eqref{eq:tildeRn}, we first show that
\be\label{eq:ets1'} 
\left|\frac{\tilde{R}_i^n-\tilde{R}_i}{\tilde{R}_i}\right| \one_{\{\tilde{R}_i > 0\}} \le C
\ee 
for some $C \in (0,\infty)$.
For $\theta \in \{\frac{1}{n},\frac{2}{n},\dotsc,1\}$, we have that 
$$
\frac{\binom{n\theta}{d}}{\binom{n}{d}} - \theta^d =\frac{n\theta(n\theta-1)(n\theta-2)\dotsm(n\theta-d+1)}{n(n-1)(n-2)\dotsm(n-d+1)}-\theta^d = \frac{a_d(n)\theta^d + a_{d-1}n^{d-1}\theta^{d-1} + \dotsb +a_1n\theta}{n(n-1)(n-2)\dotsm(n-d+1)},
$$
where
$$a_d(n)=n^{d}-n(n-1)(n-2)\dotsm(n-d+1)$$
is a polynomial of $n$ with degree $d-1$, and $a_k$ is a constant only depending on $k$, for $k=d-1,d-2,\dotsc,1$.
For example,
$$a_1=(-1)(-2)\dotsm(-d+1).$$
Therefore, letting $q_i = \sum_{j \ge i} \tilde{\eta}_j$, replacing $\theta$ with $q_i$ and $q_{i+1}$, and then taking difference, we have
\begin{align*}
	\left|\frac{\tilde{R}_i^n-\tilde{R}_i}{\tilde{R}_i}\right| \one_{\{\tilde{R}_i > 0\}} & = \frac{\left|a_d(n)(q_i^d-q_{i+1}^d)+\sum_{k=1}^{d-1}a_kn^k(q_i^k-q_{i+1}^k)\right|}{n(n-1)(n-2)\dotsm(n-d+1)} \frac{1}{q_i^d-q_{i+1}^d} \one_{\{q_i^d-q_{i+1}^d > 0\}} \\
	& \le \frac{a_d(n)}{\Pi_{i=0}^{d-1} (n-i)} + \sum_{k=1}^{d-1}|a_k| \frac{n^k}{\Pi_{i=0}^{d-1} (n-i)} \frac{q_i^k-q_{i+1}^k}{q_i^d-q_{i+1}^d} \one_{\{q_i^d-q_{i+1}^d > 0\}}.
\end{align*}
For the first term on the right hand side, 
$$\frac{a_d(n)}{\Pi_{i=0}^{d-1} (n-i)} \le \frac{C}{n}.$$
For the second term on the right hand side and each $k=1,2,\dotsc,d-1$,
\begin{equation*}
	\frac{n^k}{\Pi_{i=0}^{d-1} (n-i)} \frac{q_i^k-q_{i+1}^k}{q_i^d-q_{i+1}^d} \one_{\{q_i^d-q_{i+1}^d > 0\}} \le \frac{n^k}{\Pi_{i=0}^{d-1} (n-i)} \frac{(1/n)^k-0^k}{(1/n)^d-0^d} = \frac{n^d}{\Pi_{i=0}^{d-1} (n-i)} \le C,
\end{equation*}
where the first inequality follows from Lemma \ref{lemma:monotone} below.
Combining these estimates gives \eqref{eq:ets1'}.

From \eqref{eq:ets1'} we have
$$\left| \frac{(R^n_i(s))^{2q}}{(R_i(s))^{2q-1}} - R_i(s) \right| \one_{\{R_i(s) > 0\}} = \left| \left( \frac{R^n_i(s)-R_i(s)}{R_i(s)} + 1 \right)^{2q} - 1 \right| R_i(s) \one_{\{R_i(s) > 0\}} \le C_{2q} R_i(s).$$
Then
\begin{align*}
	& a^2(n) \log \E\left[ \exp\left( \sum_{i=0}^\infty \int_0^T \left( \frac{(R^n_i(s))^{2q}}{(R_i(s))^{2q-1}} - R_i(s) \right)ds\right) \right] \\
	& \le a^2(n) \log \E [\exp(\sum_{i=0}^\infty \int_0^T C_{2q} R_i(s) \,ds)] = a^2(n) C_{2q} T \to 0
\end{align*}
as $n \to \infty$.
This completes the proof.
\end{proof}

\begin{Lemma}\label{lemma:monotone}
	Fix integers $d > k \ge 1$.
	Then the function $f(x,y):=\frac{x^k-y^k}{x^d-y^d}$ is decreasing in $x$ and $y$ for $0 \le y < x \le 1$.
\end{Lemma}

\begin{proof}
	First consider the function $g(z) = dz^k-kz^d-d+k$, $z \in (0,\infty)$.
	Clearly $g'(z)=dk(z^{k-1}-z^{d-1}$ is positive for $0 < z < 1$ and negative for $z > 1$. 
	Since $g(1)=0$, we have $g(z) < 0$ for all $z \in (0,1) \cup (1,\infty)$.
	
	Now note that
	$$\frac{\partial f}{\partial x} = \frac{kx^{k-1}(x^d-y^d) - (x^k-y^k)dx^{d-1}}{(x^d-y^d)^2} = \frac{x^{k+d-1}}{(x^d-y^d)^2} g(\frac{y}{x}), \quad \frac{\partial f}{\partial y} = \frac{y^{k+d-1}}{(y^d-x^d)^2} g(\frac{x}{y}).$$
	Therefore $\frac{\partial f}{\partial x} < 0$ and $\frac{\partial f}{\partial y} < 0$ for $0 < y < x \le 1$.
	The remaining boundary case $y=0$ is trivial, as $f(x,0)=x^{k-d}$ is decreasing in $x>0$.
	This completes the proof.
\end{proof}

\bibliographystyle{plain}

\begin{thebibliography}{10}

\bibitem{aghajani2019hydrodynamic}
Reza Aghajani and Kavita Ramanan.
\newblock The hydrodynamic limit of a randomized load balancing network.
\newblock {\em The Annals of Applied Probability}, 29(4):2114--2174, 2019.

\bibitem{banerjee2019join}
Sayan Banerjee and Debankur Mukherjee.
\newblock Join-the-shortest queue diffusion limit in halfin?whitt regime:
  Tail asymptotics and scaling of extrema.
\newblock {\em The Annals of Applied Probability}, 29(2):1262--1309, 2019.

\bibitem{bramson2012asymptotic}
Maury Bramson, Yi~Lu, and Balaji Prabhakar.
\newblock Asymptotic independence of queues under randomized load balancing.
\newblock {\em Queueing Systems}, 71:247--292, 2012.

\bibitem{braverman2020steady}
Anton Braverman.
\newblock Steady-state analysis of the join-the-shortest-queue model in the
  halfin--whitt regime.
\newblock {\em Mathematics of Operations Research}, 45(3):1069--1103, 2020.

\bibitem{brightwell2018supermarket}
Graham Brightwell, Marianne Fairthorne, and Malwina~J Luczak.
\newblock The supermarket model with bounded queue lengths in equilibrium.
\newblock {\em Journal of Statistical Physics}, 173:1149--1194, 2018.

\bibitem{BudhirajaDupuisGanguly2015moderate}
A.~Budhiraja, P.~Dupuis, and A.~Ganguly.
\newblock {Moderate deviation principles for stochastic differential equations
  with jumps}.
\newblock {\em The Annals of Probability}, 44(3):1723--1775, 2016.

\bibitem{BudhirajaWu2017moderate}
A.~Budhiraja and R.~Wu.
\newblock {Moderate deviation principles for weakly interacting particle
  systems}.
\newblock {\em Probability Theory and Related Fields}, 168(3):721--771, 2017.

\bibitem{budhiraja2019diffusion}
Amarjit Budhiraja and Eric Friedlander.
\newblock Diffusion approximations for load balancing mechanisms in cloud
  storage systems.
\newblock {\em Advances in Applied Probability}, 51(1):41--86, 2019.

\bibitem{BudhirajaFriedlanderWu2019many}
Amarjit Budhiraja, Eric Friedlander, and Ruoyu Wu.
\newblock Many-server asymptotics for join-the-shortest-queue: Large deviations
  and rare events.
\newblock {\em The Annals of Applied Probability}, 31(5):2376--2419, 2021.

\bibitem{BudhirajaMukherjeeWu2019supermarket}
Amarjit Budhiraja, Debankur Mukherjee, and Ruoyu Wu.
\newblock Supermarket model on graphs.
\newblock {\em The Annals of Applied Probability}, 29(3):1740--1777, 2019.

\bibitem{cardinaels2022power}
Ellen Cardinaels, Sem Borst, and Johan~SH van Leeuwaarden.
\newblock Power-of-two sampling in redundancy systems: The impact of assignment
  constraints.
\newblock {\em Operations Research Letters}, 50(6):699--706, 2022.

\bibitem{CoppiniDietertGiacomin2019law}
Fabio Coppini, Helge Dietert, and Giambattista Giacomin.
\newblock {A law of large numbers and large deviations for interacting
  diffusions on Erd{\"{o}}s--R{\'e}nyi graphs}.
\newblock {\em Stochastics and Dynamics}, 20(02):2050010, 2020.

\bibitem{Der2022scalable}
Mark~Van der Boor, Sem~C Borst, Johan~SH Van~Leeuwaarden, and Debankur
  Mukherjee.
\newblock Scalable load balancing in networked systems: A survey of recent
  advances.
\newblock {\em SIAM Review}, 64(3):554--622, 2022.

\bibitem{EschenfeldtGamarnik18}
Patrick Eschenfeldt and David Gamarnik.
\newblock Join the shortest queue with many servers. the heavy-traffic
  asymptotics.
\newblock {\em Mathematics of Operations Research}, 43(3):867--886, 2018.

\bibitem{graham2005functional}
Carl Graham.
\newblock Functional central limit theorems for a large network in which
  customers join the shortest of several queues.
\newblock {\em Probability theory and related fields}, 131:97--120, 2005.

\bibitem{IkedaWatanabe1990SDE}
N.~Ikeda and S.~Watanabe.
\newblock {\em {Stochastic Differential Equations and Diffusion Processes}},
  volume~24 of {\em North-Holland Mathematical Library}.
\newblock Elsevier, 1981.

\bibitem{luczak2006maximum}
Malwina~J Luczak and Colin McDiarmid.
\newblock On the maximum queue length in the supermarket model.
\newblock {\em The Annals of Probability}, pages 493--527, 2006.

\bibitem{luczak2005strong}
Malwina~J Luczak and James Norris.
\newblock Strong approximation for the supermarket model.
\newblock {\em Annals of Applied Probability}, pages 2038--2061, 2005.

\bibitem{Mitzenmacher01}
Michael Mitzenmacher.
\newblock {The power of two choices in randomized load balancing}.
\newblock {\em IEEE Trans. Parallel Distrib. Syst.}, 12(10):1094--1104, 2001.

\bibitem{MukherjeeBorstLeeuwaardenWhiting18}
Debankur Mukherjee, Sem~C. Borst, Johan S.~H. van Leeuwaarden, and Philip~A.
  Whiting.
\newblock Universality of power-of-d load balancing in many-server systems.
\newblock {\em Stochastic Systems}, 8(4):265--292, 2018.

\bibitem{RuttenMukherjee2022load}
Daan Rutten and Debankur Mukherjee.
\newblock {Load balancing under strict compatibility constraints}.
\newblock {\em Math. Oper. Res.}, 2022.

\bibitem{VvedenskayaDobrushinKarpelevich96}
Nikita~Dmitrievna Vvedenskaya, Roland~L'vovich Dobrushin, and
  Fridrikh~Izrailevich Karpelevich.
\newblock {Queueing system with selection of the shortest of two queues: An
  asymptotic approach}.
\newblock {\em Problemy Peredachi Informatsii}, 32(1):20--34, 1996.

\end{thebibliography}

\end{document}